\documentclass[11pt]{article}
\usepackage[dvips]{epsfig}
\usepackage{amsmath}
\usepackage{amsfonts}
\usepackage{amssymb}
\usepackage{graphicx}
\usepackage{latexsym,amssymb,amsmath,amsthm,amsfonts}
\usepackage{psfrag}
\usepackage{bbm}
\usepackage{float}
\usepackage{latexsym}
\usepackage{color}
\newtheorem{lem}{Lemma}[section]
\newtheorem{thm}{Theorem}[section]
\newtheorem{cor}{Corollary}[section]
\newtheorem{prop}{Proposition}[section]
\newtheorem{remark}{Remark}[section]

\newtheorem{defn}{Definition}[section]
\numberwithin{equation}{section}

\newcommand{\be}{\begin{equation}}
\newcommand{\ee}{\end{equation}}
\newcommand{\bi}{\bibitem}
\newcommand{\ben}{\begin{enumerate}}
\newcommand{\een}{\end{enumerate}}
\newcommand{\beq}{\begin{eqnarray}}
\newcommand{\eeq}{\end{eqnarray}}
\newcommand{\beqn}{\begin{eqnarray*}}
\newcommand{\eeqn}{\end{eqnarray*}}

\title{Linearized heat semigroups on Finsler measure spaces and some applications}

\author{\small{Qiaoling Xia}\\
{\small {\it Department of Mathematics, School of Sciences }}\\ {\small{\it Hangzhou Dianzi University}}\\
 {\small {\it Hangzhou, 310018, Zhejiang Province, P.R.China}}\\
 {\small {\it E-mail address:  xiaqiaoling$@$hdu.edu.cn}}}

\date{}
\begin{document}
\maketitle{}

\begin{abstract}
 It is known that the Finsler heat flow is a nonlinear flow. This leads to the study of linearized heat semigroups for the Finsler heat flow. In this paper, we give the properties of linearized heat semigroups and prove that the semigroup is conservative on complete Finsler measure spaces $(M, F, m)$ with weighted Ricci curvature Ric$_N$ bounded from below. As applications, we give new proofs of Li-Yau's inequalities established in \cite{Xia2} and \cite{OS2} respectively in the compact case and extend them to the complete Finsler measure spaces with Ric$_N\geq K$ for $K\in \mathbb R$.  Finally we give several equivalent characterizations of Ric$_\infty\geq K (K\in \mathbb R)$ via the linearized heat semigroup approach and their applications.


\end{abstract}

\section{Introduction}

  It is well known that Li-Yau's inequalities for linear parabolic equations are essential tools to study global geometry and analysis on Riemannian manifolds (\cite{LY}, \cite{Da}, \cite{BGL}). Every Li-Yau type inequality implies the corresponding Harnack inequality, which leads to the estimates of the heat kernel, the Green function, the first eigenvalue and various entropy formulas with monotonicity for heat equations, etc. (\cite{Li}, \cite{LX}). In recent years, Li-Yau's inequality has been generalized in different settings, such as weighted Riemannian manifolds (\cite{Li}), Alexandrov spaces and RCD$^*(K, N)$-metric measure spaces (\cite{QZZ}, \cite{ZZ}) and so on. In these cases, the heat flows are linear and the proofs of these inequalities are based on the Bochner technique and the maximum principle.

  Recall that an $n$-dimensional {\it Finsler manifold} $(M, F)$ means an $n$-dimensional smooth manifold $M$ equipped with a Finsler metric (or Finsler structure) $F: TM\rightarrow [0, +\infty)$ such that $F_x=F|_{T_xM}$ is a Minkowski norm on $T_xM$ at each point $x\in M$, where $TM$ is the tangent bundle of $M$.  Given a smooth measure $m$, the triple $(M, F, m)$ is called a {\it Finsler metric measure space} (simply, Finsler measure space).
A Finsler measure space is not a metric space in the usual sense because $F$ may be nonreversible, i.e., $F(x, v)\neq F(x, -v)$ may happen. This non-reversibility causes the asymmetry of the associated distance function. Since $F$ is not smooth at the zero section and non-quadratic, the gradient $\nabla u$ of a smooth function $u$ on $M$ is only continuous and nonlinear on $M$. From the variational viewpoint, we may still define the Finsler Laplacian $\Delta_m $ (simply, denoted by $\Delta$) with respect to the measure $m$ (\cite{Sh}). But it is a fully nonlinear divergence differential operator.

 S. Ohta and K.-T. Sturm introduced a nonlinear heat flow $\partial_t u=\Delta u$ (in the weak sense) on an $n$-dimensional Finsler measure space and applied the classical technique due to Saloff-Coste (\cite{SC}) to show the unique existence and a certain interior regularity of solutions to the heat flow (\cite{Oh1}, \cite{OS2}). For the elliptic case, similar results were established by Y.Ge and Z. Shen (\cite{GS}). Further, Ohta-Sturm used the maximum principle for the linearized parabolic operator of $\partial_t-\Delta$ to get Davies' estimate (which improved the well-known Li-Yau's inequality in \cite{LY}):
   \begin{eqnarray}|\nabla \log u|^2-\alpha\partial_t(\log u)\leq -\frac{N\alpha^2K'}{4(\alpha-1)}+\frac{N\alpha^2}{2t}  \label{Davis's est}\end{eqnarray}
  on a compact Finsler manifold with Ric$_N\geq K(K\in \mathbb R)$ for any constant $\alpha>1$ and $t>0$, by establishing the Bochner-Weitzenb\"ock formula, where $N\in [n, \infty)$ and $K'=\min\{K, 0\}$ (\cite{OS2}). Recently, the author of the present paper gave more general Li-Yau's inequalities for the heat flow via establishing a more refined estimate in the compact case, and gave a local Li-Yau type gradient estimate in the complete noncompact case by Moser's iteration (\cite{Xia2}).  Moser's iteration indeed works for the case when $\alpha$ is a constant greater than $1$ in (\ref{Davis's est}) but seems not to work for the other cases. A natural question is how to establish a more general Li-Yau type inequality on a complete and noncompact Finsler measure space.

Due to the nonlinearity of Finsler Laplacian and the non-reversibility of the Finsler distance function,  the linear elliptic or parabolic theory of PDEs cannot be directly applied to the Finsler setting. Because of this, we try to exploit some nonlinear approaches to study geometric analysis and topology on Finsler manifolds (\cite{Oh1}, \cite{Xia3}). For example, S. Ohta introduced the linearized heat semigroup $\{P_{s, t}^{\nabla u}\}$ for the heat equation $\partial_tu=\Delta u$ (in the weak sense) and used this to show some functional inequalities, such as the Poincar\'e inequality, the log--Sobolev inequality, and the Bakry--Ledoux's isoperimetric inequality etc. on Finsler measure spaces with positive weighted Ricci-infinity curvature bounds (\cite{Oh1}). The author of the present paper gave the gradient estimates for the $p$--eigenfunctions (resp. $p$--harmonic functions) via Moser's iteration (\cite{Xia1}).  In this paper, we shall further develop the linearized heat semigroup theory and pursue its applications.
To ensure the existence of $\{P_{s, t}^{\nabla u}\}$ on Finsler measure spaces $(M, F, m)$, we need to assume that $F$ satisfies  {\it uniform convexity and uniform smoothness}, i.e., there exist uniform and positive constants $\kappa^*$, $\kappa$ such that for any $x\in M$, $v\in T_xM$ and $Y\in TM\setminus\{0\}$, we have
\begin{eqnarray}\kappa^*F^2(x, v)\leq g_Y(v, v)\leq \kappa F^2(x, v), \label{unif cs cons}\end{eqnarray} where $g_Y=(g_{ij}(Y))$ is the weighted Riemannian metric on $M$ induced by $F$ with the reference vector $Y$ (see \S 2). (\ref{unif cs cons}) implies that $0<\kappa^*\leq 1\leq \kappa$ and $F$ has finite reversibility $\Lambda$ (see (\ref{Lambda}) below).
The uniform smoothness and the uniform convexity were first introduced in Banach space theory by Ball, Carlen and Lieb in \cite{BCL}. In fact there are many Finsler metrics satisfying (\ref{unif cs cons}) (see \cite{Xia1}). If $(M, F, m)$ is compact (without boundary),  then (\ref{unif cs cons}) holds and $m(M)<\infty$ trivially. The semigroup $\{P_{s, t}^{\nabla u}\}$ for the heat flow is always conservative in this case (see Corollary \ref{cor31} in \S 3). For the complete and noncompact case, we have the following result.

\begin{thm} \label{thm31} Let $(M, F, m)$ be an $n$-dimensional complete and noncompact Finsler measure space satisfying (\ref{unif cs cons}). If the volume of $M$ is finite, i.e., $m(M)<\infty$,  or  Ric$_N\geq K$ for some $N\in [n, \infty]$ and $K\in \mathbb R$ when $m(M)$ is infinite, then the linearized heat semigroup $\{P_{s, t}^{\nabla u}\}$ is conservative.
In particular, if $f\in H_0^1(M)$ and satisfies $k_1\leq f\leq k_2$ almost everywhere on $M$ for $k_1, k_2\in \mathbb R$, then $k_1\leq P_{s, t}^{\nabla u}(f)\leq k_2$ almost everywhere on $M$. Moreover, $\left(P^{\nabla u}_{s, t}(f)\right)^2\leq P_{s, t}^{\nabla u}(f^2)$. \end{thm}

As the first application of Theorem \ref{thm31}, we give new proofs of Li-Yau's inequalities established in \cite{Xia2} and \cite{OS2} respectively in the compact case and extend them to the compact or complete noncompact Finsler measure spaces with Ric$_N\geq K$ for any $K\in \mathbb R$. To state our result, let  $a(t)$ be a positive $C^1$ function on $(0, \infty)$ satisfying

(A1) $\lim\limits_{t\rightarrow 0+}a(t)=0$ and $\lim\limits_{t\rightarrow 0+}\frac{a(t)}{a'(t)}=0$;

(A2) $\frac{a'(t)^2}{a(t)}$ is continuous and integrable on the interval $[0, T]$ for all $T>0$.

For any $K\in \mathbb R$, define
\beq \alpha(t):&=& 1-\frac{2K}{a(t)}\int_0^ta(s)ds,\label{alpha} \\
\varphi(t):&=&-\frac{NK}2+\frac{NK^2}{2a(t)}\int_0^ta(s)ds+\frac N{8a(t)}\int_0^t\frac{a'(s)^2}{a(s)}ds.\label{varphi}\eeq It is easy to see that $(\alpha-1)'+(\log a)'(\alpha-1)+2K=0$ and $\varphi'+(\log a)'\varphi-\frac N8\left((\log a)'-2K\right)^2=0$
with $\lim\limits_{t\rightarrow 0+}\alpha(t)=1$ and $\lim\limits_{t\rightarrow 0+}\varphi(t)=+\infty$. Obviously, $\alpha(t)\equiv 1$ if and only if $K=0$.
Note that (\ref{alpha}) excludes the case when $\alpha(t)$ is a constant, which is not equal to $1$.

\begin{thm} \label{thm11} Let $(M, F, m)$ be an $n$-dimensional compact Finsler measure space and $u(t, x)$ be a positive global solution to $\partial_t u=\Delta u$ (in the weak sense) on $[0, \infty)\times M$. Assume that Ric$_N\geq K$ for some $N\in [n, \infty)$ and $K\in \mathbb R$. Then, for any positive $C^1$ function $a(t)$ on $(0, \infty)$ satisfying (A1)-(A2), we have
\beq F^2(\nabla(\log u))-\alpha(t)\partial_t(\log u) \leq \varphi(t)\label{ap}\eeq on $(0, \infty)\times M$, where $\alpha(t)$ and $\varphi(t)$ are defined by (\ref{alpha})-(\ref{varphi}). \end{thm}

If Ric$_N\geq -K(K\geq 0)$ and $a(t)$ satisfies $a'(t)>0$ as well as (A1)-(A2), then Theorem \ref{thm11} is reduced to Theorem 1.2 in \cite{Xia2} obtained in a different way. In fact, we can drop the restriction that $a'(t)>0$ (see (A1)). For the complete and noncompact case, we have the following.

\begin{thm} \label{thm41} Let $(M, F, m)$ be a complete and noncompact Finsler measure space satisfying (\ref{unif cs cons}) and $u(t, \cdot)$ be a positive global solution to $\partial_tu=\Delta u$ (in the weak sense) on $[0, \infty)\times M$ with the initial $u(0, \cdot)\in H^1(M)\cap L^\infty(M)$ and
$F^2(\nabla u)\in H^1(M)$ for all $t>0$.  Assume that Ric$_N\geq K$ for some $N\in [n, \infty)$ and $K\in \mathbb R$. Then we have (\ref{ap}) on $(0, \infty)\times M$ for any positive $C^1$ function $a(t)$ on $(0, \infty)$ satisfying (A1)-(A2), where $\alpha(t)$ and $\varphi(t)$ are defined by (\ref{alpha})-(\ref{varphi}). \end{thm}

The assumption that $F^2(\nabla u)\in H^1(M)$ for all $t>0$ is a technical assumption (likely redundant), which holds in the compact case thanks to the regularity of $H^2_{loc}(M)$ by Proposition \ref{prop21*} below. Since (\ref{unif cs cons}) is trivial on compact Finsler manifolds, the restrictions on $u$ and $F$ in Theorem \ref{thm41} are not necessary in the compact case (see Theorem \ref{thm11}).  For the complete and noncompact case, we have $F(\nabla u)\in L^\infty(M)$ if Ric$_N\geq -K (K\geq 0)$  by Corollary 1.5 in \cite{Xia2}. Thus $F^2(\nabla u)\in L^2(M)$ from the regularity of $u$ (see Proposition \ref{prop21*} below). In this case, the assumption that $F^2(\nabla u)\in H^1(M)$ can be replaced by $F(\nabla [F(\nabla u)])\in L^2(M)$.
Further, Theorems  \ref{thm11} and \ref{thm41} imply various new Li-Yau type inequalities on compact or complete noncompact Finsler manifolds in both positive and negative curvature by choosing suitable $a(t)$. For example, if we take $a(t)=t^2$ for $t>0$, by (\ref{alpha})-(\ref{varphi}) and Theorems \ref{thm11} and \ref{thm41}, then we have
\begin{cor}\label{cor11} Under the same assumptions as in Theorem \ref{thm11} or \ref{thm41}, we have
\beq F^2(\nabla\log u)-\left(1-\frac 23 Kt\right)\partial_t(\log u)\leq -\frac 12NK\left(1-\frac K3t\right)+\frac N{2t}. \label{ap-1} \eeq
 In particular, when Ric$_N\geq 0$, we have
 \beq F^2(\nabla\log u)-\partial_t(\log u)\leq \frac N{2t}.\label{ap-0} \eeq \end{cor}

 Obviously, (\ref{ap-0}) is exactly a generalization of sharp Li-Yau's inequality on complete Riemannian manifolds with nonnegative Ricci curvature.  If we take $a(t)=4\tau_1\sin^2(\tau_1t)$ for $\tau_1>0$ and $t\in (0, \pi/\tau_1)$ and $a(t)=4\tau_2\sinh^2(\tau_2t)$ for $\tau_2, t>0$, then we obtain the following corollary which follows from (\ref{alpha})-(\ref{varphi}) and Theorems \ref{thm11}--\ref{thm41}.
\begin{cor} \label{cor12} Let $(M, F, m)$ and $u(t, \cdot)$ be the same as in Theorem \ref{thm11} or \ref{thm41}. Assume that Ric$_N\geq K$, where $K\neq 0$. Then

(1) for any $\tau_1>0$ and $t\in (0, \pi/\tau_1)$, we have
\beq F^2(\nabla\log u)- \alpha_1(t) \partial_t(\log u) \leq\varphi_1(t), \label{Linearized-Li-Yau's est-1} \eeq  where
\beqn \alpha_1(t):&=& 1-\frac {K}{2\tau_1 \sin^2(\tau_1 t)}\left(2\tau_1 t-\sin(2\tau_1 t)\right), \\
\varphi_1(t):&=&\frac N2\left\{-K+\frac {\tau_1}{4\sin^2(\tau_1 t)}\left(\sin(2\tau_1 t)+2\tau_1 t\right) +\frac {K^2}{4\tau_1\sin^2(\tau_1 t)}\left(2\tau_1 t-\sin(2\tau_1 t)\right)\right\}.\eeqn

(2) for any $\tau_2>0$ and $t>0$, we have
\beq  F^2(\nabla\log u)-\alpha_2(t)\partial_t(\log u)\leq \varphi_2(t), \label{Linearized-Li-Yau's est-2}\eeq  where
\beqn \alpha_2(t):&=& 1-\frac {K}{2\tau_2 \sinh^2(\tau_2 t)}\left(\sinh(2\tau_2 t)-2\tau_2 t\right), \\
\varphi_2(t):&=&\frac N2\left\{-K+\frac {\tau_2}{4\sinh^2(\tau_2 t)}\left(\sinh(2\tau_2 t)+2\tau_2 t\right)+\frac {K^2}{4\tau_2\sinh^2(\tau_2 t)}\left(\sinh(2\tau_2 t)-2\tau_2 t\right)\right\}.\eeqn \end{cor}

In particular, taking
$\tau_2=|K|>0$ in (\ref{Linearized-Li-Yau's est-2}) yields Li-Xu's inequality on complete Riemannian manifolds with Ric$\geq -k(k\geq 0)$ (\cite{LX}).
 \begin{cor} \label{cor13}  Let $(M, F, m)$ and $u(t, \cdot)$ be the same as in Theorem \ref{thm11} or \ref{thm41}.  Assume that Ric$_N\geq K (K\neq 0)$ for some $N\in [n, \infty)$. Then we have (\ref{ap}) in which
 \beqn \alpha(t)&=&1-\frac{\sinh (2Kt)-2Kt}{2\sinh^2(Kt)}, \ \ \ \ \varphi(t)=-\frac {NK}2(1-\coth(Kt)). \label{LX-F}\eeqn  Its linearized version is exactly (\ref{ap-1}).
 \end{cor}

As an application of Theorems \ref{thm11} and \ref{thm41}, we obtain the Harnack inequality by following the proof of Corollary 1.1 in \cite{Xia2}.
 \begin{cor}  Under the same assumptions as in Theorem \ref{thm11} or \ref{thm41}, we have
\beq u(t_1, x_1) \leq u(t_2, x_2) \exp \left(\frac{ d_F^2(x_2, x_1)}{4(t_2-t_1)^2}\int_{t_1}^{t_2}\alpha(t)dt+\int_{t_1}^{t_2}\frac{\varphi(t)}{\alpha(t)}dt\right) \label{H-ineq-1}  \eeq for all  $0<t_1<t_2< \infty$ and $x_1, x_2\in M$. \end{cor}

 Note that Theorems \ref{thm11} and \ref{thm41} exclude the case when $\alpha(t)$ is a constant which is not equal to $1$. For the case when $\alpha(t)=\alpha$, which is a constant greater than $1$,  Ohta-Sturm obtained Davies' gradient estimate (\ref{Davis's est}) in the compact case and the author of the present paper obtained the Li-Yau type gradient estimate in the complete and noncompact case for positive solutions to the heat flow (\cite{OS2}, \cite{Xia2}).  More generally, we obtain a generalized Li-Yau inequality via the linearized heat semigroup approach (see Theorem \ref{thm51} below). As applications, we generalize Ohta-Sturm's estimate in both the compact case and the complete noncompact case as follows.

\begin{thm} \label{thm12} Let $(M, F, m)$ and $u(t, \cdot)$ be the same as in Theorem \ref{thm11} or \ref{thm41}. Assume that Ric$_N\geq -K (K>0)$.  Then there is a number $\chi_0\in (1, 1+\frac{\pi^2}{K^2t^2})$ such that $\partial_t(\log u)\geq -\frac {NK}4\chi_0,$ and
\beq F^2(\nabla\log u)-\alpha\partial_t(\log u)\leq \frac{N\alpha^2}{2t}+\frac{NK\alpha^2}{4(\alpha-1)} \label{DE*}\eeq for any constant $\alpha>1$ and $t>0$.
\end{thm}

 For the case when Ric$_N\geq K>0$, we have a different type of gradient estimate (see Corollary \ref{cor51} below). Finally we give several equivalent characterizations of Ric$_\infty\geq K$ for some $K\in \mathbb R$ via the linearized heat semigroup approach (see Theorem \ref{thm61} below).

\section{Preliminaries}
In this section, we briefly review some basic concepts in Finsler geometry.  For more details, we refer to  \cite{Sh}, \cite{Oh1} and \cite{Xia3}.
\subsection{Finsler metrics and curvatures}
Let $(M, F)$ be an $n$-dimensional Finsler manifold and $TM$ (resp. $T^*M$) be the tangent (resp. cotangent) bundle. We say that a Finsler metric $F$ is {\it reversible} if $F(x, -y)=F(x, y)$ for any $(x, y)\in TM$. Otherwise, we define the {\it reversibility} $\Lambda$ of $F$ by
$$\Lambda:= \sup\limits_{(x, v)\in TM\setminus\{0\}}\frac{ F(x, -v)}{F(x, v)}.$$ Obviously, $\Lambda\in [1, \infty]$ and $\Lambda=1$ if and only if $F$ is reversible.
In local coordinates $(x^i, y^i)$, let $g_{ij}(x, y):=\frac 12[F^2(x, y)]_{y^iy^j}$. Then $g_y=g_{ij}(x, y)dx^i\otimes dx^j$ is called the {\it fundamental tensor} of $F$. Given a Finsler metric $F$ on $M$,  its dual $F^*: T^*M\rightarrow \mathbb R$ is defined by $F^*(x, \xi):=\sup_{0\neq y\in T_xM}\frac{\xi(y)}{F(x, y)}$ for $(x, \xi)\in T^*M$. Then $F^*$ is also a Finsler metric on $M$ (\cite{Sh}).

 Given a non-vanishing smooth vector field $Y$, we obtain the weighted Riemannian metric $g_Y=(g_{ij}(Y))$ on $M$ induced by $F$ with the reference vector $Y$, which is given by
\begin{eqnarray}g_Y(v, w)=g_{ij}(x, Y_x)v^iw^j, \ \ \ \ \mbox{for} \ v, w\in T_xM\ \mbox {and} \ x\in M. \label{w-metric}\end{eqnarray} Obviously, $F^2(Y)=g_Y(Y, Y)$. If $F$ satisfies (\ref{unif cs cons}), then $F$ has finite reversibility $\Lambda$, i.e.,
\begin{eqnarray}1\leq \Lambda \leq \min\{\sqrt{\kappa}, \sqrt{1/{\kappa^*}}\}.\label{Lambda}\end{eqnarray}   $F$ is Riemannian if and only if $ \kappa^*=1$ if and only if $\kappa=1$ (\cite{Oh1}).

 For $x_1, x_2\in M$, the {\it distance} from $x_1$ to $x_2$ is defined by $d_F(x_1, x_2):=\inf_{\gamma}\int_0^1F(\dot{\gamma}(t))dt$, in which the infimum is taken over all $C^1$ curves $\gamma: [0,1]\rightarrow M$ from $\gamma(0)=x_1$ to $\gamma(1)=x_2$. Note that $d_F(x_1, x_2)\neq d_F(x_2, x_1)$  unless $F$ is reversible. $d_F(p, \cdot)$ is differentiable almost everywhere on $M$ and $F(\nabla d_F)=1$ almost everywhere on $(M, F)$.  If $F$ has finite reversibility $\Lambda$, then
\beq \Lambda^{-1} d_F(q, p)\leq d_F(p, q)\leq \Lambda d_F(q, p).\label{dd}\eeq
 Now we define the {\it forward and backward geodesic balls} of radius $r$ with the center at $x\in M$ by
\begin{eqnarray*}B^+_r(x):=\{z\in M\mid d_F(x,z)<r\}, \ \ \  B^-_r(x):=\{z\in M\mid d_F(z, x)<r\}. \end{eqnarray*}

A $C^1$ curve $\gamma: [0, \ell]\rightarrow M$ is called a {\it geodesic} if it has constant speed (i.e.,
$F(\gamma, \dot\gamma)$ is constant) and if it is locally minimizing. Such a geodesic is in fact a $C^\infty$ curve.  A Finsler manifold $(M, F)$ is said to be {\it forward complete (resp. backward complete)} if each geodesic defined on $[0, \ell)$ (resp. $(-\ell, 0]$) can be extended to a geodesic defined on $[0, \infty)$ (resp. $(-\infty, 0]$). If the reversibility of $F$ is finite, then the forward completeness and the backward completeness are equivalent. We simply say that $F$ is {\it complete} in this case.

Let $TM_0=TM\setminus\{0\}$ be the slit tangent bundle of $M$ and $\pi: TM_0\rightarrow M$ be the projection map. The pull-back bundle $\pi^*TM$ on $TM_0$ admits a unique linear connection $D$, called the {\it Chern connection}. It is uniquely determined by the Finsler metric $F$ (\cite{Sh}, \cite{Xia3}).
Given a non-vanishing vector field $Y$ on $M$,  the {\it Riemannian curvature} $R^Y$ is defined by
\begin{eqnarray*}R^Y(U, V)W:=D^Y_UD^Y_V W-D^Y_V D^Y_U W-D^Y_{[U, V]}W,\end{eqnarray*} for any vector fields $U, V, W$ on $M$, where $D^Y$ is the covariant derivative defined by the Chern connection $D$ with respect to the reference vector $Y$.   $R^Y$ is independent of the choice of connections (\cite{Sh}).  For two linearly independent vectors $v, w\in T_xM\setminus\{0\}$, the {\it flag curvature} with flag $v$  is defined by
\begin{eqnarray*} K(v, w)=\frac{g_V(R^V(v, w)w, v)}{g_V(v,v)g_V(w, w)-g_V(v, w)^2},\end{eqnarray*} where $V$ is a geodesic field around $x$ with $v=V_x\in T_xM$, i.e., $D^V_VV=0$. The flag curvature $K(v, w)$ coincides with the sectional curvature of the $2$-plane spanned by $v, w$ with respect to the Riemannian metric $g_V$ in a neighborhood of $x$. The {\it Ricci curvature} is defined by
\begin{eqnarray*}Ric(v):=\sum_{i=1}^{n-1}K(v, e_i),\end{eqnarray*} where $e_1, \cdots, e_{n-1}, \frac v{F(v)}$ form an orthonormal basis of $T_xM$ with respect to $g_V$.

\begin{defn}[\cite{Oh1}]\label{def21}
Given a unit vector $v\in T_{x}M$, let $\gamma:(-\varepsilon,\varepsilon)\longrightarrow M$ be the geodesic with $\gamma(0)=x$ and $\dot\gamma(0)=v$. We set $dm =e^{-\Psi(\gamma(t))}\textmd{vol}_{\dot\gamma}$ along $\gamma$, where $\textmd{vol}_{\dot\gamma}$ is the volume form of $g_{\dot\gamma}$. Define the weighted  Ricci curvature involving a parameter $N\in(n,\infty)$ by
 $$Ric_N(v):=Ric(v)+{(\Psi\circ \gamma)^{''}(0)}-\frac{(\Psi\circ \gamma)^{\prime}(0)^{2}}{N-n},$$ where Ric is the Ricci curvature of $F$. Also define Ric$_\infty(v):=Ric(v)+{(\Psi\circ \gamma)^{''}(0)}$ and Ric$_n(v):=\lim \limits_{N\rightarrow n}{\rm Ric}_N$.
\end{defn}

We remark that $(\Psi\circ \gamma)^{\prime}(0)$ is exactly the S-curvature with respect to $m$ and  $(\Psi\circ \gamma)^{''}(0)$ is the change rate of S-curvature along $\gamma$ (\cite{Sh}, \cite{Sh1}). Moreover,  $\textmd{Ric}_N(\lambda v)=\lambda^2\textmd{Ric}_N(v)$ for any $\lambda\geq 0$ and $N\in[n, \infty]$.  We say that $\textmd{Ric}_N\geq K$ for some $K\in {\mathbb R}$ if $\textmd{Ric}_N(x, v)\geq KF^2(x, v)$ for all $(x, v)\in TM$.

\subsection{The gradient and Finsler Laplacian}
Let ${\mathcal L}: TM \rightarrow T^*M$ be the {\it Legendre transform} associated with the Finsler metric $F$ and its dual $F^*(x, \xi)$ on $M$, that is, a transformation $\mathcal L$ from $T_xM$ to $T_x^*M$ for any $x\in M$ defined by
\beqn {\mathcal L}( y):=\left\{\begin{array}{ll} g_{y}(y, \cdot),  \ \ \ \ \ \ \ \ \ \ \ \ y\in T_xM\setminus\{0\}, \\
 0, \ \ \ \ \ \ \ \ \ \ \ \ \ \ \ \ \ \ \ \  y=0.\end{array}\right. \eeqn
 One can check that $\mathcal L$ is a diffeomorphism from $TM\setminus \{0\}$ onto $T^*M\setminus \{0\}$ and norm-preserving, namely, $F(x, y)=F^*(x, \mathcal L(y))$ for any $(x, y)\in TM$. Consequently, $g^{ij}(x, y)=g^{*ij}(x, \xi)$, where $(g^{ij}(x, y))=(g_{ij}(x, y))^{-1}$ and $g^{*ij}(x, \xi):=\frac 12\left(F^{*2}(x, \xi)\right)_{\xi_i\xi_j}$ with  $\xi: =\mathcal L(y)=\xi_idx^i\in T^*_xM\setminus\{0\}$ (\cite{Sh}). It is easy to see that $F^{*2}(\xi)=g^{*ij}(\xi)\xi_i\xi_j$ by homogeneity. If $F$ satisfies (\ref{unif cs cons}), then  $F^*$ also satisfies uniform convexity and uniform smoothness with
\begin{eqnarray}\tilde \kappa^*F^{*2}(x, \eta)\leq g^*_{\xi}(\eta, \eta)\leq \tilde \kappa F^{*2}(x, \eta)\label{unif cons}\end{eqnarray} for any $x\in M$, $\xi\in T_x^*M\setminus\{0\}$ and $\eta\in T_x^*M$, where $\tilde \kappa =(\kappa^*)^{-1}$ and $\tilde \kappa^*=\kappa^{-1}$ (See Lemma 2.1.2, \cite{Xia3}).

For any $u\in C^\infty(M)$,
the {\it gradient  vector} $\nabla u(x)$ of  $u$ is defined by $\nabla u(x)
:={\mathcal L}^{-1}(du)(x)\in  T_xM$. Obviously, $\nabla u=0$ if $du=0$. In a local coordinate system, we reexpress $\nabla u$ as
\begin{eqnarray*}
\nabla u(x):=\left\{\begin{array}{ll}
                 g^{*ij}(du(x))\frac{\partial u}{\partial x^i}\frac{\partial}{\partial x^j}  & x\in M_u,\\
                 0 & x\in M\setminus M_u,
               \end{array}\right. \label{grad-u}
          \end{eqnarray*}
where $M_u=\{x\in M|du(x)\neq 0\}$. In general, $\nabla u$ is only continuous on $M$, but smooth on $M_u$.

The {\it Hessian} of $u$ is defined by
$$\nabla^2u(X, Y)=g_{\nabla u}\left(D^{\nabla u}_X\nabla u, Y\right) $$ for any $X, Y\in TM$. It is easy to check that $\nabla^2u(X, Y)$ is symmetric with respect to $X$ and $Y$.

 Given a smooth measure $m$ on $(M, F)$,  the {\it Finsler Laplacian} $\Delta$ of a function $u\in H^1_{\rm{loc}}(M)$ is defined by
 \begin{eqnarray*}\int_M\phi\Delta udm: =-\int_M d\phi(\nabla u)dm  \label{Laplacian} \end{eqnarray*}  for any $\phi\in C^\infty_0(M)$ (\cite{Sh}). Obviously, $\Delta$ is a nonlinear divergence type differential operator.
 Locally, write $dm=\sigma(x)dx$. Then
\begin{eqnarray}\Delta u=\frac 1{\sigma}\frac{\partial}{\partial x^i}\left(\sigma g^{ij}(\nabla u)\frac{\partial u}{\partial x^j}\right) \ \ \ \ \ \ \ {\rm{on}}\ \  M_u. \label{loc-Lap}\end{eqnarray}
For any non-vanishing smooth vector field $Y$ on $M$, we define the  {\it weighted gradient vector} and the {\it weighted Laplacian} on $M$ by
\begin{eqnarray*} \nabla^{Y}u:=  g^{ij}(x, Y_x)\frac{\partial u}{\partial x^{j}}\frac{\partial }{\partial x^{i}} \label{eq nablau}\end{eqnarray*}
and $\Delta^{Y}u:=\mbox{div}_m(\nabla^{Y}u)$ in the sense of distribution. We remark that $\Delta^{\nabla u}$ is exactly the linearization of $\Delta$ on $M_u$. Moreover,  $\nabla^{\nabla u}u=\nabla u$ and $\Delta^{\nabla u}u=\Delta u$.

Let $H^1(M): =W^{1, 2}(M)$ and $H_0^{1}(M)$ be the completions of $C^\infty(M)$ and $C_0^{\infty}(M)$(the space of smooth functions with compact support) respectively under the norm
   \begin{eqnarray*}\|u\|_{H^1}:=\|u\|_{L^2(M)}  +  \frac 12\|F^*(du)\|_{L^2(M)}+  \frac 12 \| F^*(-du)\|_{L^2(M)}.\label{norm1}\end{eqnarray*}  Then $H^{1}(M)$ and $H^{1}_0(M)$ are Banach spaces with respect to the norm $\|\cdot\|_{H^1}$. Moreover, we denote by $H^1_{loc}(M)$ the space of functions in $H^1(\Omega)$ for all relatively compact open sets $\Omega\subset M$ and $H^1_c(M)$ the space of functions in $H^1_{loc}(M)$ with compact support.
 In \cite{OS2}, the authors gave pointwise Bochner-Weitzenb\"ock formulae, which play important roles in Finsler geometry. Moreover, they also gave the integrated Bochner-Weitzenb\"ock formulae with the help of the following fact to overcome the ill-posedness of $\nabla u$ on $M\setminus M_u$.
\begin{lem} \label{lem21} For each $f\in H^1_{loc}(M)$, we have $df=0$ almost everywhere on $f^{-1}(0)$. If $f\in H^1_{loc}(M)\cap L^\infty_{loc}(M)$, then $d(f^2/2)=fdf=0$ also holds almost everywhere on $f^{-1}(0)$.\end{lem}

Lemma \ref{lem21} implies that $\Delta u=0$ almost everywhere on $M\setminus M_u$.

\begin{prop}[\cite{Oh1}, \cite{OS2}]\label{thm22}  Given $u\in H^2_{loc}(M)\cap C^1(M)\cap H^1(M)$ with $ \Delta u\in H^1_{loc}(M)$ ,  we have
 \begin{eqnarray} -\int_M d\phi \left[\nabla^{\nabla u}\left(\frac{F^2(\nabla u)}2\right) \right]dm\geq  \int_M \phi\left\{d(\Delta u)(\nabla u)+ \textmd{Ric}_{\infty}(\nabla u)\right\}dm\label{Bochner3''} \end{eqnarray} and
\begin{eqnarray}-\int_M d\phi\left[\nabla^{\nabla u}\left(\frac{F^2(\nabla u)}2 \right)\right]dm \geq \int_M \phi\left\{d(\Delta u)(\nabla u)+ \textmd{Ric}_N(\nabla u) +\frac{(\Delta u)^2}{N}\right\}dm\label{Bochner4''}\end{eqnarray} for  $N\in [n, \infty]$ and all nonnegative functions $\phi\in H^1_c(M)\cap L^\infty(M)$.\end{prop}

\begin{remark} \label{rem21} {\rm In Proposition \ref{thm22}, the test function $\phi\in H_c^1(M)\cap L^\infty(M)$ can be extended to $\phi\in H_0^1(M)\cap L^\infty(M)$ under the extra assumption that $\Delta u\in H^1(M)$ and \beq F\big(\nabla^{\nabla u}(F^2(\nabla u))\big)\in L^2(M).\label{u-L}\eeq
 In fact, if $\phi\in H_0^1(M)\cap L^\infty(M)$, then it follows from $\Delta u\in H^1(M)$ and $u\in H^1(M)$ that the integrals of the right hand side (RHS) in (\ref{Bochner3''})-(\ref{Bochner4''}) are well defined. The integrability of LHS in (\ref{Bochner3''})-(\ref{Bochner4''})  can be seen from the following inequality
\beqn & &\int_M\mid(d\phi-d\tilde\phi)\big(\nabla^{\nabla u}F^2(\nabla u)\big)\mid dm \nonumber \\
& &\leq \int_M\max\big\{F^*\big(d(\phi-\tilde \phi)\big), F^*\big(d(\tilde\phi-\phi)\big)\big\}F\big(\nabla^{\nabla u}F^2(\nabla u)\big)dm\nonumber \\
 & &\leq  \left(\|F^*(d(\phi-\tilde\phi))\|_{L^2(M)}+\|F^{*}(d(\tilde\phi-\phi))\|_{L^2(M)}\right)\cdot \|F\big(\nabla^{\nabla u}F^2(\nabla u)\big)\|_{L^2(M)} \eeqn for any $\tilde\phi\in C_0^\infty(M)$. The rest follows from the arguments by approximation.} \end{remark}

\subsection {The heat flow and linearized heat semigroups}

Recall that a function $u(t, x)$ on $[0, T]\times M$, $T>0$ ($T$ may be the infinity),  is a {\it global  solution} to the heat flow $\partial_tu=\Delta u$ if $u\in L^2\left([0, T], H_0^1(M)\right)\cap H^1\left([0, T], H^{-1}(M)\right)$ and if, for every $t\in [0, T]$ and $\phi\in C_0^\infty(M)$, it holds that
  \beq \int_M \phi \partial_tudm\ = -\int_M d\phi(\nabla u)dm. \label{heat-eq}\eeq The test function $\phi$ can be taken from $H_0^1(M)$ by similar arguments as in Remark \ref{rem21}.    The following existence and regularity for the heat flow were due to Ohta-Sturm.
   \begin{prop} [\cite{Oh1}, \cite{OS1}] \label{prop21*} Assume that $\Lambda<\infty$. Then $\mathrm{(i)}$\  for each initial datum $u_0\in H^1_0(M)$ and $T>0$, there exists a unique global solution $u(t, x)$ for $t\in[0, T]$ to the heat equation, and the distributional Laplacian $ \Delta u(t, x)$ is absolutely continuous with respect to $m$ for all $t\in (0, T)$.

 $\mathrm{(ii)}$\  One can take the continuous version of a global solution $u$, and it enjoys the $H^2_{loc}$-regularity in $x$ as well as the $C^{1, \beta}$-regularity in both $t$ and $x$ for some $0<\beta<1$. Moreover, $\partial_tu$ lies in $H^1_{loc}(M)\cap C(M)$, and further in $H^1_0(M)$ if $F$ has a finite uniform smoothness constant $\kappa$. The elliptic regularity shows that $u$ is $C^\infty$ on $\cup_{t>0}(\{t\}\times M_{u(t, x)})$. \end{prop}

  For each $s\geq 0$ and $T>0$, given a non-vanishing measurable vector field $V_t$ on $[s, s+T]\times M$, we define a linearized heat equation $\partial_tv=\Delta^{V_t}v$ in the sense that
  \begin{eqnarray} \int_s^{s+T}\int_M v \partial_t\phi dm dt=\int_s^{s+T}\int_Md\phi(\nabla^{V_t}v)dmdt \label{v-he}\end{eqnarray} for $v\in H_0^1(M)$ and $\phi\in C_0^\infty((s, s+T)\times M)$. Similarly, $\phi$ can also be taken from $H_0^1((s, s+T)\times M)$ as before.  This can be regarded as a linear heat equation on an evolving weighted Riemannian manifold $(M, g_{V_t}, m)$. By the classical results on linear differential operators involving a parameter $t$, we have the following existence and regularity result.
  \begin{prop}\label{prop22} Assume that $(M, F, m)$ is a Finsler measure space satisfying (\ref{unif cs cons}). Then

  (i) for each $s\geq 0, T>0$ and $f\in H_0^1(M)$, there exists a unique solution $v(t, \cdot)$ to $\partial_tv=\Delta^{V_t}v$ in the sense of (\ref{v-he}) with $v(s, x)=f(x)$. Moreover, $v$ lies in $L^2\left([s, s+T], H_0^1(M)\right)\cap H^1\left([s, s+T], H^{-1}(M)\right)$ as well as $C([s, s+T], L^2(M))$.

  (ii) the solution $v(t, x)$ in (i) is H\"older continuous on $(s, s+T)\times M$.
  \end{prop}

The proof follows from Theorem 11.3 in \cite{RR} (also see Proposition 13.18 in \cite{Oh1} or Proposition 3.1 in \cite{Oh2}).
 Given a global solution $u(t, x)\in H_0^1(M)$ to the heat equation $\partial_tu=\Delta u$ for $t\geq 0$, we fix a family of non-vanishing measurable vector fields $V_t$ such that $V_t=\nabla u(t, \cdot)$ on $M_{u_t}$. Proposition \ref{prop22} implies that for every $s\geq 0, T>0$ and $t\in [s, s+T]$, there exists a uniquely determined linearized operator $P_{s, t}^{\nabla u}: H_0^1(M)\rightarrow H_0^1(M)$ such that for every $f\in H_0^1(M)$, the unique weak solution $v(t, x)$ to $\partial_tv=\Delta^{V_t}v$ with $v(s, \cdot)=f$ is given by  $v(t, x)=P_{s, t}^{\nabla u}(f)(x)$. Thus we obtain the following linearized heat equation (in the weak sense),  introduced by Ohta-Sturm in \cite{OS2}.
 \begin{eqnarray}\partial_t[P_{s, t}^{\nabla u}(f)]=\Delta^{V_t}[P_{s, t}^{\nabla u}(f)], \ \ \ \ P_{s, s}^{\nabla u}(f)=f. \label{LHE}\end{eqnarray}  Here we will suppress the dependence on the choice of $V_t$ in $P_{s, t}^{\nabla u}$ by abuse of notation.

 Since $\Delta^{V_t}$ depends on $t$, $P_{s, t}^{\nabla u}$ is linear but not self-adjoint. S. Ohta (cf. \cite{Oh1}) also defines the adjoint heat semigroup $\{\hat P_{t, s}^{\nabla u}(\psi)\}_{s\in[0, t]}\subset H_0^1(M)$ for any $\psi\in H_0^1(M)$ and $t>0$ as the weak solution to the equation
  \begin{eqnarray} \partial_s[\hat P_{t, s}^{\nabla u}(\psi)]=-\Delta^{V_s}[\hat P_{t, s}^{\nabla u}(\psi)], \ \ \ \ \hat P_{t, t}^{\nabla u}(\psi)=\psi,  \label{AHE}\end{eqnarray} or equivalently,
  \begin{eqnarray*} \partial_s[\hat P_{t, t-s}^{\nabla u}(\psi)]=\Delta^{V_{t-s}}[\hat P_{t, t-s}^{\nabla u}(\psi)], \ \ \ s \in [0, t].\end{eqnarray*}
  It is easy to see that
   \begin{eqnarray} \int_M \psi  P_{s, t}^{\nabla u}(f) dm=\int_M f\hat P_{t, s}^{\nabla u}(\psi) dm. \label{PP}\end{eqnarray}

\section{Properties of the linearized heat semigroup}

In this section, we study the properties of the linearized heat semigroup as preparation. For this, we always assume that $F$ satisfies (\ref{unif cs cons}) to ensure the existence of the linearized heat semigroup.

 \begin{prop}\label{prop31} The family of linear operators $\{P_{s, t}^{\nabla u}\}_{0\leq s<t}$ enjoys the following properties.

(1) $u(t, x)=P_{s, t}^{\nabla u}(u(s, \cdot))(x)$ for any global solution $u(t,x)$ to $\partial_tu=\Delta u$ and $0\leq s<t<\infty$.

(2) (Semigroup property) \ $P_{\sigma, t}^{\nabla u}\circ P_{s, \sigma}^{\nabla u}=P_{s, t}^{\nabla u}$ for any $s\leq \sigma<t$.

(3) (Contraction property)\  $\|P_{s, t}^{\nabla u}\|_{L^2}\leq 1$. In particular, $P_{s, t}^{\nabla u}(0)=0$. Further, $\|P_{s, t}^{\nabla u}\|_{L^p}\leq 1$ for any $p\in [1, \infty)$.

(4)(Boundedness)\  If $0\leq f\leq k$ almost everywhere on $M$, then $0\leq P_{s,t}^{\nabla u}(f)\leq k$ almost everywhere on $M$ for all $t>s$. Further, assume that $(M, F)$ is complete and $m(M)<\infty$. Then the linearized heat semigroup is conservative, i.e., $P_{s, t}^{\nabla u}(1)=1$. In this case, if $k_1\leq f\leq k_2$ almost everywhere on $M$ for $k_1, k_2\in \mathbb R$, then $k_1\leq P_{s, t}^{\nabla u}(f)\leq k_2$ almost everywhere on $M$.
\end{prop}

\begin{proof} (1)  For any $\psi\in C^\infty_0(M)$, observe that
$$\frac d{ds}\int_Mu(s, \cdot)\hat P_{t, s}^{\nabla u}(\psi)dm= \int_M \left\{-d(\hat P_{t, s}^{\nabla u}(\psi))(\nabla u(s, \cdot))+du(\nabla^{V_s}\hat P_{t, s}^{\nabla u}(\psi))\right\}dm =0$$ by (\ref{heat-eq}) and (\ref{AHE}). From this and (\ref{PP}), we have $$\int_M \psi(x)u(t,x)dm=\int_M u(s,x)\hat P_{t, s}^{\nabla u}(\psi)(x)dm=\int_M \psi(x)P_{s, t}^{\nabla u}(u(s, \cdot))(x)dm, $$ which implies that $u(t, x)=P_{s, t}^{\nabla u}(u(s, \cdot))(x)$ a.e. on $M$ by arbitrariness of $\psi$. Since  $u(t, x)$ and $P_{s, t}^{\nabla u}(u(s, \cdot))$ are continuous, (1) holds on $M$.

 (2) For any $f\in H_0^1(M)$, $P_{s, t}^{\nabla u}(f)$ is the unique solution to (\ref{LHE}) at the time $t$ by Proposition \ref{prop22}. On the other hand, let $\tilde f: = P_{s, \sigma}^{\nabla u}(f)\in H_0^1(M)$, which is the unique solution to (\ref{LHE}) at $\sigma\in [s, t)$. Taking this $\tilde f$ as an initial value, we get a solution  $P_{s, t}^{\nabla u}(\tilde f)$ to (\ref{LHE})$_1$ at the time $t$. Thus $P_{\sigma, t}^{\nabla u}\circ P_{s, \sigma}^{\nabla u}(f)=P_{s, t}^{\nabla u}\big(P_{s, t}^{\nabla u}(f)\big)$ is a solution to (\ref{LHE})$_1$ at the time $t$. Obviously, $P_{\sigma, t}^{\nabla u}\circ P_{s, \sigma}^{\nabla u}(f)=f$ when $t=s$. By uniqueness of solutions and the arbitrariness of $f$, we get the conclusion.

 (3) Let $v(t, \cdot): =P_{s, t}^{\nabla u}(f)$. Then $v(t, \cdot)$ is a solution to (\ref{LHE}) by definition. By the divergence lemma, we have
\begin{eqnarray*} \frac d{dt}(\|v(t, \cdot)\|^2_{L^2})=-2\int_M\|\nabla^{V_t}v\|^2_{g_{V_t}}dm\leq 0,\end{eqnarray*} which means that $\|v(t, \cdot)\|_{L^2}$ is decreasing in $t$. In particular, $\|P_{s, t}^{\nabla u}(f)\|_{L^2}\leq \|f\|_{L^2}$ for $0\leq s<t<\infty$, which means that $\|P_{s, t}^{\nabla u}\|_{L^2}\leq 1$ and $\|P_{s, t}^{\nabla u}(0)\|_{L^2}\leq 0$. Obviously, $\|P_{s, t}^{\nabla u}(0)\|_{L^2}$ is nonnegative. Thus $P_{s, t}^{\nabla u}(0)=0$.  Further,  $\|P_{s, t}^{\nabla u}\|_{L^p}\leq 1$ for any $p\in [1, \infty)$ by interpolation.

(4)  Let $v(t, \cdot)=P_{s, t}^{\nabla u}(f)$ and $\bar v(t, \cdot): =\min\{\max\{v(t, \cdot), 0\}, k\}$ for any fixed $k\geq 0$. Then $0\leq \bar v\leq v$ if $v\geq 0$ and $\bar v=0$ if $v\leq 0$.
Obviously, $d\bar v=0$ on the set $\{v\leq 0\}$. By Lemma 7.6 in \cite{GT} (also see the proof of Theorem 1.1 in \cite{ZX}), we have $d\bar v=0$  on $\{v\geq k\geq 0\}$ and $d\bar v=dv$ on $\{0\leq v<k\}$ in the weak sense.
  Consequently, $\bar v\in H_0^1(M)$ and it is a solution of (\ref{LHE})$_1$ with $\bar v(s, x)=f(x)$. Moreover,
\begin{eqnarray*}& & \int_M v(t, \cdot)(v-\bar v)(t, \cdot)dm\\
& &=\int_s^t\frac d{dr}\left(\int_Mv(r, \cdot)(v-\bar v)(r,\cdot)dm\right)dr \\
& & =\int_s^t\int_M\left\{-2\|\nabla^{V_r}v\|^2_{g_{V_r}}+d\bar v(\nabla^{V_r}v)+dv(\nabla^{V_r}\bar v)\right\}dm dr\leq 0,\end{eqnarray*}
where we used $d\bar v(\nabla^{V_r}v)=dv(\nabla^{V_r}\bar v)\leq \|\nabla^{V_r}v\|_{g_{V_r}}^2$ in the last inequality. Since $v(v-\bar v)\geq 0$ everywhere and $v(v-\bar v)>0$ when $v\neq \bar v$, the above inequality means that $v=\bar v$ and hence $0\leq v\leq k$ a.e. on $M$.

  Since $(M, F)$ is complete and $F$ satisfies (\ref{unif cs cons}), which implies that the reversibility of $F$ is finite, we have $H_0^1(M)=H^1(M)$ by Lemma 11.4 in \cite{Oh1}. In particular, if $m(M)<\infty$, then $1\in H_0^1(M)$ and we have the mass conservation: \beqn \int_M u_tdm=\int_M u_s dm \eeqn for any global solution $u_t: =u(t, \cdot)$ to the heat equation and all $0\leq s<t<\infty$ by (\ref{heat-eq}) with $\phi=1$. This means that $P^{\nabla u}_{s, t}(1)=1$. Consequently, $P_{s, t}^{\nabla u}(f-k)=P_{s, t}^{\nabla u}(f)-k$ for any constant $k$. The last assertion follows from the first one by using $f-k_i$ ($i=1, 2$) instead of $f$.
 \end{proof}

Note that the adjoint heat semigroup satisfies the linearized heat equation backward in time. Therefore the analogues of Proposition \ref{prop22} and Proposition \ref{prop31} hold for $\{\hat P_{t, t-s}^{\nabla u}\}_{s\in [0, t]}$. Since $\|P_{s, t}^{\nabla u}(f)\|_{L^2}^2$ is non-increasing in $t$, it is uniquely extended to linear contractive semigroups acting on $L^2(M)$. In fact, (3) in Proposition \ref{prop31} implies that $P_{s, t}^{\nabla u}$  can be extended to a contraction operator on $L^p(M)$ for all $p\in [1, \infty)$ and $s<t$. Thus we can regard $\{P_{s, t}^{\nabla u}\}_{0\leq s<t}$  as a contractive heat semigroup acting on $L^p(M)$ associated with the operator $\Delta^{V_t}-\partial_t$ on an evolving weighted Riemannian manifold $(M, g_{V_t}, m)$.

\begin{prop} \label{prop32} For any $f, g\in H_0^1(M)$ and $0\leq s<t<\infty$, we have
\beq \left(P^{\nabla u}_{s, t}(fg)\right)^2\leq P_{s, t}^{\nabla u}(f^2)P_{s, t}^{\nabla u}(g^2). \label{Pfg}\eeq Further, if $(M, F)$ is complete and  $m(M)<\infty$, then $\left(P^{\nabla u}_{s, t}(f)\right)^2\leq P^{\nabla u}_{s, t}(f^2)$.  \end{prop}

\begin{proof}  For any $\lambda\in \mathbb R$, it follows from Proposition \ref{prop31} that
\beqn 0\leq P_{s, t}^{\nabla u}\left((\lambda f+g)^2\right)=\lambda^2 P_{s, t}^{\nabla u}(f^2)+2\lambda P_{s, t}^{\nabla u}(fg)+P_{s, t}^{\nabla u}(g^2).\eeqn Hence the discriminant  $\left(P_{s, t}^{\nabla u}(fg)\right)^2-P_{s, t}^{\nabla u}(f^2)P_{s, t}^{\nabla u}(g^2)$ is nonpositive, which implies (\ref{Pfg}). Further, if $(M, F)$ is complete and  $m(M)<\infty$, then $P^{\nabla u}_{s, t}(1)=1$ by  Proposition \ref{prop31}(4). Letting $g(x)=1$ in (\ref{Pfg})  yields $(P^{\nabla u}_{s, t}(f))^2\leq P^{\nabla u}_{s, t}(f^2)$.  \end{proof}

If $(M, F, m)$ is compact, then $F$ satisfies (\ref{unif cs cons}) and $m(M)<\infty$. It follows from Propositions \ref{prop31} and \ref{prop32} that
\begin{cor} \label{cor31} The linearized heat semigroup $\{P_{s, t}^{\nabla u}\}$ on a compact Finsler measure space is  conservative, i.e., $P_{s, t}^{\nabla u}(1)=1$. In particular, if $k_1\leq f\leq k_2$ almost everywhere on $M$ for $k_1, k_2\in \mathbb R$, then $k_1\leq P_{s, t}^{\nabla u}(f)\leq k_2$ almost everywhere on $M$. Moreover, $\left(P^{\nabla u}_{s, t}(f)\right)^2\leq P^{\nabla u}_{s, t}(f^2)$ for any $f\in H^1(M)$.   \end{cor}

 For a complete and noncompact Finsler measure space $(M, F, m)$, we need a volume growth estimate of geodesic balls to study the conservativeness of $\{P_{s, t}^{\nabla u}\}$. Now we define the exponential map $\exp_x: T_xM\rightarrow M$ by $\exp_xv:=\gamma(1)$ for any $v\in T_xM$ and $x\in M$ if there is a geodesic $\gamma: [0, 1]\rightarrow M$ with $\gamma(0)=x$ and $\gamma'(0)=v$. For any $x\in M$ and a unit vector $v\in T_xM$,  let $r(v)\in (0, \infty]$ be the supremum of $r>0$ such that the geodesic $\gamma(t)=\exp_x(tv)$ is minimal on $[0, r]$. If $r(v)<\infty$, then $\exp_x(r(v)v)$ is called a {\it cut point} of $x$, and the cut locus $\mathcal C_x$ of $x$ is defined as the set of all cut points of $x$, which has zero measure. The exponential map $\exp_x$ is a $C^\infty$-diffeomorphism from $\{tv|v\in T_x M, F(v) = 1, t\in (0, r(v))\}$ to $D_x:=M\setminus (\mathcal C_x\cup {x})$ (\cite{Sh}).

 For any $x\in M$ and $z\in D_{x}$, we choose the geodesic polar coordinates $(r, \theta)$ centered at $x$ such that the distance function $r(z)=d_F(x, z)=F(v)$ and $\theta^{\alpha}(z)=\theta^\alpha(\frac v{F(v)})$, where $v=\exp^{-1}(z)\in T_{x}M\setminus\{0\}$. Note that $r(z)$ is smooth on $D_{x}$ and $\nabla r$ is a geodesic vector field with $F(\nabla r)=1$. By the Gauss lemma, the unit radial coordinate vector $\frac {\partial}{\partial r}$ is orthogonal to coordinate vectors $\frac{\partial}{\partial \theta^\alpha}$ with respect to $g_{\nabla r}$ for $1\leq \alpha\leq n-1$ (\cite{Sh}). Write $dm|_{\exp_{x}(rv_0)}=\sigma( r, \theta)dr d\theta$, where $v_0=\frac v{F(v)}\in I_{x}:=\{v\in T_{x}M| F(v)=1\}$. By (\ref{loc-Lap}), we have
\beq \Delta r=\frac{\partial}{\partial r}(\log \sigma).\label{Lap-r}\eeq
 Let $B_R^+(x)$ be the forward geodesic ball of radius $R$ on $(M, F)$. Set $\mathcal D_r(x): =\{v_0\in I_{x}|rv_0\in \exp^{-1}(D_{x}\cap {B}_{r}^+(x))\}$. Obviously, $\mathcal D_t(x)\subset \mathcal D_s(x)$ for any $0<s<t<R$. The volume  of $B_R^+(x)$ with respect to $m$ is given by
\begin{eqnarray*}m(B_R^+(x))=\int_{B_R^+(x)}dm=\int_{B_R^+(x)\cap D_{x}}dm=\int_0^Rdr\int_{\mathcal D_r(x)}\sigma( r, \theta)d\theta. \label{dm}\end{eqnarray*}

 If $(M, F, m)$ is a forward complete Finsler manifold with Ric$_N\geq K$ for any $N\in [n, \infty)$ and $K\in \mathbb R$, then we have
\beq m(B_R^+(x))\leq m(B_1^+(x))R^Ne^{R\sqrt{(N-1)|K|}} \label{V-B-N}\eeq for any $x\in M$ and $R>1$ (Proposition 5.1, \cite{Xia}). Next we consider the case when Ric$_\infty\geq K$.

\begin{lem}\label{lem31}  Let $(M, F, m)$ be an $n$-dimensional forward complete Finsler measure space and $m_0:=\sup\limits_{z\in r^{-1}(r_0)}\Delta r(z)$ whenever the distance function $r(z)= d_F(x, z)$ is smooth and $r(z)>r_0$ for some $r_0>0$. Assume that Ric$_\infty\geq K$ for some $K\in \mathbb R$. Then, for any $x\in M$, the following inequalities hold for $R>2r_0$,
\beq  m(B^+_{R}(x))\leq \left\{\begin{array}{ll} C m\left(B^+_{2r_0}(x)\right), &K>0, \\
C m\left(B^+_{2r_0}(x)\right) \exp\left\{-\frac {K}4\left(R-r_0-\frac{m_0}K\right)^2\right\}, & K<0, \end{array}\right.
\label{m-Br} \eeq
 and
 \beq  m(B^+_{R}(x))\leq \left\{\begin{array}{ll}  C m\left(B^+_{2r_0}(x)\right) R, & m_0=0, \\
C m\left(B^+_{2r_0}(x)\right)e^{m_0R}, & m_0\neq 0, \end{array}\right. \label{m-Br-0}\eeq when $K=0$, where $C=C(r_0, m_0, K)$ is a positive constant depending on $r_0, m_0, K$. In particular, $m(M)<\infty$ when $K>0$.
\end{lem}
\begin{proof} By the assumptions and Theorem 13.2 in \cite{Sh1}, the distance function $r(z)=d_F(x, z)$ on $D_x\setminus B^+_{r_0}(x)$ satisfies
\beq \Delta r(z)\leq m_0-K\left(r(z)-r_0\right), \label{r-Delta}\eeq  where $m_0:=\sup\limits_{z\in r^{-1}(r_0)}\Delta r(z)$. Combining (\ref{Lap-r}) and (\ref{r-Delta}) together yields
\beq \frac{\partial}{\partial r}(\log \sigma)\leq m_0-K\left(r(z)-r_0\right) \label{sigma-der} \eeq pointwise on $D_x$ for $r(z)>r_0$.

If $K\neq 0$, then integrating (\ref{sigma-der}) on both sides in $r$ from $s\in (r_0, t)$ to $t$ gives
\beqn \log\frac{\sigma(t, \theta)}{\sigma(s, \theta)} &\leq & m_0(t-s)-\frac K2\left((t-r_0)^2-(s-r_0)^2\right) \nonumber\\
&=& -\frac K2 \left(t-r_0-\frac {m_0}K\right)^2+\frac K2 \left(s-r_0-\frac {m_0}K\right)^2,  \eeqn which implies that
\beqn \exp\left\{-\frac K2\left(s-r_0-\frac {m_0}K\right)^2\right\}\sigma(t, \theta)\leq \exp\left\{-\frac K2\left(t-r_0-\frac {m_0}K\right)^2\right\}\sigma(s, \theta).  \eeqn  Integrating both sides of this inequality in $t$ from $r_1\in (s, t)$ to $r_2\in (t, \infty)$ and then in $s$ from $r_0$ to $r_1$ yield \beqn & & \int_{r_0}^{r_1}\exp\left\{-\frac K2\left(s-r_0-\frac {m_0}K\right)^2\right\}ds \int_{r_1}^{r_2}\sigma(t, \theta)dt\nonumber \\
 & &\leq \int_{r_1}^{r_2}\exp\left\{-\frac K2\left(t-r_0-\frac {m_0}K\right)^2\right\}dt \int_{r_0}^{r_1}\sigma(s, \theta)ds.
 \eeqn
 Further, integrating both sides of this inequality on $\mathcal D_t(x)$ and using $\mathcal D_t(x)\subset \mathcal D_s(x)$ for any $0<s<t<R$, one obtains that
\beqn & &\frac{m(B^+_{r_2}(x))-m(B^+_{r_1}(x))} {m(B^+_{r_1}(x))-m(B^+_{r_0}(x))} \nonumber \\
& & \leq \int_{r_1}^{r_2}\exp\left\{-\frac K2\left(t-r_0-\frac {m_0}K\right)^2\right\}dt \Big/ \int_{r_0}^{r_1}\exp\left\{-\frac K2\left(s-r_0-\frac {m_0}K\right)^2\right\}ds \nonumber \\
& & \leq  -1+ \int_{r_0}^{r_2}\exp\left\{-\frac K2\left(t-r_0-\frac {m_0}K\right)^2\right\}dt\Big /\int_{r_0}^{r_1}\exp\left\{-\frac K2\left(t-r_0-\frac {m_0}K\right)^2\right\}dt.\eeqn Thus we have
\beq & &\frac{m(B^+_{r_2}(x))-m(B^+_{r_0}(x))} {m(B^+_{r_1}(x))-m(B^+_{r_0}(x))}\nonumber \\
& &\leq  \int_{r_0}^{r_2}\exp\left\{-\frac K2\left(t-r_0-\frac {m_0}K\right)^2\right\}dt \Big /\int_{r_0}^{r_1}\exp\left\{-\frac K2\left(t-r_0-\frac {m_0}K\right)^2\right\}dt. \label{mm-B}\eeq
Note that
\beqn \left(\int_0^re^{s^2/2}ds\right)^2 &=&\int_0^r\int_0^re^{(s^2+t^2)/2}ds dt \nonumber \\
&\leq &\int_0^{\frac{\pi}2}d\theta\int_0^{\sqrt{2}r}e^{{\rho^2}/2}\rho d\rho=\frac{\pi}{2}\left(e^{r^2}-1\right),\eeqn where we used the polar coordinate system $(\rho, \theta)$ in the second inequality. In the same way,
\beqn \left(\int_0^re^{s^2/2}ds\right)^2 &=&\int_0^r\int_0^re^{(s^2+t^2)/2}ds dt \nonumber \\
&\geq &\int_0^{\frac{\pi}2}d\theta\int_0^{r}e^{{\rho^2}/2}\rho d\rho=\frac{\pi}{2}\left(e^{r^2/2}-1\right).\eeqn
Consequently,
\beq \sqrt{\frac{\pi}{2}\left(e^{r^2/2}-1\right)}\leq \int_0^re^{s^2/2}ds\leq \sqrt{\frac{\pi}{2}\left(e^{r^2}-1\right)}\leq \sqrt{\frac{\pi}{2}}e^{r^2/2}. \label{e-integral-1} \eeq
Similarly, we have
\beq \sqrt{\frac{\pi}{2}\left(1-e^{-r^2/2}\right)}\leq \int_0^re^{-s^2/2}ds\leq \sqrt{\frac{\pi}{2}\left(1-e^{-r^2}\right)}\leq \sqrt{\frac{\pi}{2}}. \label{e-integral-2} \eeq
Taking $r_2: =R>r_1:=2r_0$ in (\ref{mm-B}), and using $m(B_{r_0}^+(x)) < m(B_{2r_0}^+(x))$ and (\ref{e-integral-1})--(\ref{e-integral-2}), one obtains
\beqn  m(B^+_{R}(x)) &\leq & m(B_{r_0}(x))+\widetilde C m(B_{2r_0}(x))\int_{r_0}^{r_2}\exp\left\{-\frac K2\left(t-r_0-\frac {m_0}K\right)^2\right\}dt \nonumber \\
&\leq & 2\widetilde C m(B^+_{2r_0}(x))\int_{r_0}^{r_2}\exp\left\{-\frac K2\left(t-r_0-\frac {m_0}K\right)^2\right\}dt\nonumber \\
&\leq & 2\widetilde C m(B^+_{2r_0}(x))\int_{0}^{R}\exp\left\{-\frac K2\left(t-r_0-\frac {m_0}K\right)^2\right\}dt\nonumber \\
&\leq & \sqrt{2\pi} \widetilde C m\left(B^+_{2r_0}(x)\right)
\eeqn
when $K>0$ and similarly,
\beqn  m(B^+_{R}(x)) \leq \sqrt{2\pi} \widetilde C m\left(B^+_{2r_0}(x)\right)\exp\left\{-\frac K2\left(R-r_0-\frac {m_0}K\right)^2\right\}
\eeqn  when $K<0$, where $\widetilde C=\left(\int_{r_0}^{2r_0}\exp\left\{-\frac K2\left(t-r_0-\frac {m_0}K\right)^2\right\}dt\right)^{-1}$, which is a positive constant depending on $r_0, m_0$ and $K$.

If $K=0$, integrating (\ref{sigma-der}) on both sides from $s\in (r_0, t)$ to $t$ gives
\beqn e^{m_0s}\sigma(t, \theta)\leq e^{m_0t}\sigma(s, \theta).  \eeqn
Similar to the previous arguments, we get
\beqn \frac{m(B^+_{r_2}(x))-m(B^+_{r_0}(x))} {m(B^+_{r_1}(x))-m(B^+_{r_0}(x))}  \leq \frac{\int_{r_0}^{r_2}e^{m_0t}dt}{\int_{r_0}^{r_1}e^{m_0t}dt}. \eeqn
Letting $r_2: =R>r_1:=2r_0$ yields
\beqn  m(B^+_{R}(x))\leq \left\{\begin{array}{ll}  \frac R{r_0}m\left(B^+_{2r_0}(x)\right), & m_0=0, \\
\left(1+\frac {e^{m_0R}-e^{m_0r_0}}{e^{2m_0r_0}-e^{m_0r_0}}\right) m\left(B^+_{2r_0}(x)\right), & m_0\neq 0, \end{array}\right.\eeqn which implies (\ref{m-Br-0}). \end{proof}

\begin{remark} {\rm In \cite{St}, the author proved that lower bounds for the Ricci curvature imply upper bound estimates for the volume growth of concentric balls on a general metric measure space. It is worth mentioning that Finsler spaces are not metric measure spaces in the usual sense since the distance function induced by Finsler metric is not symmetric in general. }
\end{remark}

Proposition \ref{prop31} shows that $P_{s, t}^{\nabla u}(1)=1$ if $(M, F)$ is complete and $m(M)<\infty$ (cf. Corollary \ref{cor31}).  For the complete Finsler measure space $(M, F, m)$ with infinite volume $m(M)$, the constant function $1$ cannot be in $H_0^1(M)$. Hence $(M, F)$ is noncompact.  In this case, we define $P_{s,t}^{\nabla u}(1)$ as the (increasing) limit of $P_{s, t}^{\nabla u}(f_i)$ as $i\rightarrow \infty$, where $\{f_i\}_{i\geq 1}$ is any sequence of positive functions in $H_0^1(M)$ increasing to $1$. We say that the linearized heat semigroup $\{P_{s, t}^{\nabla u}\}$ on $(M, F, m)$ is {\it conservative} if $P_{s, t}^{\nabla u}(1)=1$ in this sense.

\begin{prop} \label{prop33} Let $(M, F, m)$ be an $n$-dimensional complete and noncompact Finsler measure space satisfying (\ref{unif cs cons}). Assume that
 \beq \int_1^\infty \frac r{\log\big(m(B_r^+(x_0))}dr=\infty \label{integral}\eeq for some $x_0\in M$.
 Then the linearized heat semigroup $\{P_{s, t}^{\nabla u}\}$ is conservative. \end{prop}
 \begin{proof}  Assume that $\{f_i\}_{i\geq 1}$ is a sequence of positive functions in $H_0^1(M)$ increasing to $1$. Then $ P_{s, t}^{\nabla u}(f_i)\in H_0^1(M)$ are solutions of (\ref{LHE}) with $0\leq P_{s, t}^{\nabla u}(f_i)\leq 1$ by Proposition \ref{prop31}.

 Let $v_i(t, \cdot): =1-P_{s, t}^{\nabla u}(f_i)\in H^1_{loc}(M)$. Then $0\leq v_i\leq 1$ and $v_i$ satisfies $\partial_t v_i=\Delta^{V_t}v_i$ in the weak sense with  $v_i(s, \cdot)=1-f_i\geq 0$ and $\lim\limits_{i\rightarrow \infty}v_i(s, \cdot)=0$, where $V_t=\nabla u(t, \cdot)$ as before.
  Fixing $s\geq 0$ and $T>s$, we have
\beq  \int_{\tau-\delta}^\tau\int_M\phi \partial_tv_i dm dt = -\int_{\tau-\delta}^\tau\int_M d\phi\left(\nabla^{V_t}v_i(t, \cdot)\right)dm dt \label{vfp}\eeq for any $\tau\in (s, T)$, $\delta\in (0, \tau-s]$ and $0\leq \phi\in H_c^1((s, T)\times M)$.

  Given a large $R>0$, let $r(x): =(d_F(x_0, x)-R)_+$ be the distance function from the geodesic ball $B_R:=B_R^+(x_0)$, where $d_F(x_0, \cdot)$ be the distance function from $x_0$ induced by $F$.  We define $$\psi(t, x):=- \frac{\kappa^{*2}r(x)^2}{4\kappa\Lambda^2(\tau+\delta -t)}$$ for $t\in [\tau-\delta, \tau]$, where $\kappa, \kappa^*$ are uniform constants and $\Lambda$ is the reversibility of $F$. Since $r(x)$ is a Lipschitz function and $F(dr)=F(\nabla r)\leq 1$ a.e. on $M$, we have $F(\nabla \psi)=F^*(d\psi)\leq \frac {\kappa^{*2}r(x)}{2\kappa\Lambda (\tau+\delta-t)}$ a.e. on $M$. Thus
 \beq \partial_t\psi+\frac{\kappa}{\kappa^{*2}}F^2(\nabla \psi)\leq 0\label{psi-t}\eeq a.e. on $M$.
 Let $\eta$ be a cut-off function on $M$ with $0\leq \eta\leq 1$ defined by
\beqn \eta(x)=\left\{\begin{array}{lll} 1 &\quad {\rm{on}}\quad B_R, \\
\frac{2R-d_F(x_0, x)}R &\quad {\rm{on}} \quad B_{2R}\setminus B_R, \\
0 &\quad{\rm{on}} \quad M\setminus B_{2R}. \end{array}\right.\label{eta}\eeqn Since $F$ has finite reversibility $\Lambda$ by (\ref{Lambda}),  $F^*(d\eta)\leq \frac \Lambda {R}$ a.e. on $M$. Moreover, by  (\ref{unif cons}), we have
\beq  \tilde\kappa^*F^2(\nabla v_i)=\tilde\kappa^* F^{*2}(d v_i)\leq \|\nabla^{V_t}v_i\|^2_{HS(V_t)}\leq \tilde\kappa F^{*2}(d v_i)=\tilde\kappa F^2(\nabla v_i),\label{ineq-1} \eeq where $\tilde\kappa^*=1/\kappa$, $\tilde\kappa=1/\kappa^*$, and $\|\cdot\|_{HS(V_t)}$ is the Hilbert-Schmidt norm with respect to the weighted Riemannian metric $g_{V_t}$ induced by $F$. Hence
  \beq |d\eta(\nabla^{V_t}v_i)|\leq \|\nabla^{V_t}\eta\|_{HS(V_t)}\cdot \|\nabla^{V_t}v_i\|_{HS(V_t)}\leq\tilde\kappa F^*(d\eta)F(\nabla v_i) \label{ineq-2} \eeq a.e. on $M$. Taking $\phi_i =v_i\eta^2 e^{\psi}$ for each $i$. Note that $\eta e^\psi\in$ Lip$_0(M)\subset H_c^1(M)$ and $v_i\in H_{loc}^1(M)$ for any fixed $t\in [\tau-\delta, \tau]$, where Lip$_0(M)$ is the set of Lipschitz functions on $M$ with compact support. Hence $\phi_i\in H_c^1([\tau-\delta, \tau]\times M)$. It follows from (\ref{vfp})-(\ref{ineq-2}) that
 \beq & & \int_{\tau-\delta}^\tau\int_M (v_i\eta^2 e^{\psi}) \partial_tv_i dmdt  \nonumber \\
 & &= -\int_{\tau-\delta}^\tau\int_M\left\{\eta^2e^\psi\|\nabla^{V_t}v_i\|^2_{HS(V_t)}+2v_i\eta e^\psi d\eta(\nabla^{V_t}v_i)+v_i\eta^2e^\psi d\psi(\nabla^{V_t}v_i) \right\} dm dt\nonumber \\
 & &\leq \int_{\tau-\delta}^\tau\int_M\left\{-\tilde\kappa^*\eta^2e^\psi F^2(\nabla v_i)+\left(\frac {\tilde\kappa^*}2F^2(\nabla v_i)\eta^2+\frac 2{\tilde\kappa^*}{\tilde\kappa^2}v^2_iF^{*2}(d\eta)\right)e^\psi\right\} dm dt\nonumber \\
 & & \ \ \ \ +\int_{\tau-\delta}^\tau\int_M\left(\frac {\tilde\kappa^*}2 F^2(\nabla v_i)+\frac 1{2\tilde\kappa^*}\tilde\kappa^2v_i^2F^{*2}(d\psi)\right)\eta^2e^\psi dm dt\nonumber \\
  & & = \int_{\tau-\delta}^\tau\int_M\frac{\tilde\kappa^2}{\tilde\kappa^*}\left(2 F^{*2}(d\eta)+\frac 12\eta^2F^{*2}(d\psi)\right)v_i^2e^\psi dm dt. \label{P-fi}\eeq
On the other hand, the LHS of (\ref{P-fi})
\beqn \int_{\tau-\delta}^\tau\int_M (v_i\eta^2 e^{\psi}) \partial_tv_i dm dt &=& \frac 12\int_{\tau-\delta}^\tau\int_M (\eta^2 e^\psi)\partial_t(v_i^2) dm dt \nonumber \\
&=& \frac 12\int_{\tau-\delta}^\tau\int_M \partial_t(v_i^2\eta^2 e^\psi) dm dt -\frac 12\int_{\tau-\delta}^\tau\int_M (v_i^2\eta^2 e^\psi)\partial_t\psi dm dt\nonumber \\
&=& \frac 12\int_M \left(v_i^2\eta^2 e^\psi\right)\Big\vert_{\tau-\delta}^\tau dm -\frac 12 \int_{\tau-\delta}^\tau\int_M (v_i^2\eta^2 e^\psi)\partial_t\psi dm dt. \label{F-fi*}
\eeqn
From  this, (\ref{P-fi}) and (\ref{psi-t}) with $\tilde \kappa^*=1/\kappa$ and $\tilde\kappa=1/\kappa^*$, one obtains
\beq \int_M \left(v_i^2\eta^2 e^\psi\right)\Big\vert_{\tau-\delta}^\tau dm &\leq & \int_{\tau-\delta}^\tau\int_M (v_i^2\eta^2 e^\psi)\partial_t\psi dm dt\nonumber \\
& & +\int_{\tau-\delta}^\tau\int_M\frac{\tilde\kappa^2}{\tilde\kappa^*}\Big(4 F^{*2}(d\eta)+\eta^2F^{*2}(d\psi)\Big)v_i^2e^\psi dm dt\nonumber \\
&\leq & \frac{4\tilde\kappa^2}{\tilde\kappa^*}\int_{\tau-\delta}^\tau\int_M v_i^2e^\psi F^{*2}(d\eta) dm dt. \label{F-fi**}\eeq   For $x\in B_{2R}\setminus B_R$ and $t\in (\tau-\delta, \tau)$, we have
$r(x)\geq R$ and $\tau+\delta-t\leq 2\delta$. Consequently,
$$\psi(t, x)\leq -\frac {\kappa^{*2}R^2}{8\delta \kappa\Lambda^2}<0.$$ Thus $\psi$ is zero in $\bar B_R$ and negative outside $B_R$, which means that $e^{\psi}=1$ in $\bar B_R$ and $e^\psi<1$ outside $B_R$. Since $0\leq v_i\leq 1$, by the choice of $\eta$ and (\ref{Lambda}),  (\ref{F-fi*}) implies
\beqn \int_{B_R} v_i^2(\tau, x) dm &\leq &\int_{B_{2R}} (v_i^2\eta^2e^\psi)(\tau, x) dm \leq  \int_{B_{2R}} v_i^2(\tau-\delta, x)dm +\frac{C}{R^2}\int_{\tau-\delta}^\tau\int_{B_{2R}\setminus B_R} e^\psi dm dt \nonumber \\
&\leq & \int_{B_{2R}} v_i^2(\tau-\delta, x) dm +\frac{C}{R^2}(T-s)m\big(B_{2R}\big)\exp\left(-\frac {\kappa^{*2}R^2}{8\delta\kappa\Lambda^2}\right).
\eeqn for some constant $C=C(\kappa, \kappa^*)$ depending on $\kappa$ and $\kappa^*$.
We can choose a small $\delta\in (0, \tau-s]$ such that $(T-s)m(B_{2R})\leq \exp\left(\frac{\kappa^*R^2}{8\delta \kappa^2\Lambda^2}\right)$, i.e.,
\beq 0<\delta\leq \frac{\kappa^{*2}R^2}{8\kappa\Lambda^2\log\big((T-s)m(B_{2R})\big)}.\label{delta}\eeq Thus,
\beq \int_{B_R} v_i^2(\tau, x)  dm \leq \int_{B_{2R}} v_i^2(\tau-\delta, x) dm +\frac{C}{R^2}.\label{iteration-1}
\eeq
Fix a large $R>1$ and $\tau\in (s, T)$, and we define a sequence of radii $\{R_k\}$ with $R_{k+1}=2R_k$, $R_0=R$ and a sequence $\{\delta_k\}\subset (0, \tau-s]$ satisfying (\ref{delta}), i.e.,
\beq 0<\delta_k\leq \frac{\kappa^{*2}R_{k}^2}{8\kappa\Lambda^2\log\big((T-s)m(B_{R_{k+1}})\big)}
=\frac{\kappa^{*2}R_{k+1}^2}{32\kappa\Lambda^2\log\big((T-s)m(B_{R_{k+1}})\big)}.\label{delta-k}\eeq
Let $\tau_0=\tau$, $\delta_0=\delta$ and $\tau_{k+1}=\tau_k-\delta_k$. The inequality (\ref{iteration-1}) gives
\beq \int_{B_{R_k}} v_i^2(\tau_k, x) dm \leq \int_{B_{R_{k+1}}} v_i^2(\tau_{k+1}, x) dm +\frac{C}{R_k^2}.\label{iteration}\eeq

{\it Case 1.}  If  $\tau_{k+1}=s$ after a finite number of iterations, i.e., $\tau=s+\sum_{i=0}^{k}\delta_i$, then
 \beqn \int_{B_R}v^2_i(\tau, x)dm &\leq &\int_{B_{R_{k+1}}}v_i^2(\tau_{k+1}, x)dm+\sum_{j=0}^k\frac{C}{R_j^2}\nonumber \\
 &\leq &\int_{B_{R_{k+1}}}v_i^2(s, x)dm+\frac{2C}{R^2}. \eeqn
  Note that $\lim_{i\rightarrow \infty}v_i(s, \cdot)=0$. Letting $i\rightarrow \infty$ and then letting $R\rightarrow \infty$ in the above inequality yield  $\lim_{i\rightarrow \infty}v_i(\tau, x)=0$ for any $x\in M$, i.e.,  $\lim_{i\rightarrow \infty}P_{s, \tau}^{\nabla u}(f_i)=1$.

 {\it Case 2.} If  $\tau_{k+1}$ reaches $s$ after an infinite number of iterations, i.e., $\tau=s+\sum_{i=0}^\infty\delta_i$, then we can choose $\{\delta_k\}\subset (0, \tau-s]$ satisfying (\ref{delta-k}) such that $\sum_{k=0}^\infty\delta_k=\infty$, which would lead to a contradiction since $\tau-s$ is finite.  In fact, let $\{\delta_k\}\subset (0, \tau-s)$ be a sequence satisfying (\ref{delta-k}). For each $k$, we may assume that
 \beqn \delta_k=\frac{\kappa^{*2}R_{k+1}^2}{8\kappa\Lambda^2\log\big((T-s)m(B_{R_{k+1}})\big)}.\eeqn Otherwise, we choose sufficiently small $\delta_{i}$ satisfying the inequality in (\ref{delta-k}) such that
 \beqn \sum_{i=0}^{\ell_1}\delta_{i}=\frac{\kappa^{*2}R_{2}^2}{8\kappa\Lambda^2\log\big((T-s)m(B_{R_2})\big)}, \ \ \ \sum_{i=t+1}^{t+\ell_k}\delta_{i}=\frac{\kappa^{*2}R_{k+1}^2}{8\kappa\Lambda^2\log\big((T-s)m(B_{R_{k+1}})\big)}, \eeqn where $t=\sum_{i=1}^{k-1}\ell_i$ for $k\geq 2$ with all $\ell_i$ being positive integers, and then argue in the same way as below by replacing $\delta_1$ and $\delta_k (k\geq 2)$ with $\sum_{i=0}^{\ell_1}\delta_{i}$  and $\sum_{i=t+1}^{t+\ell_k}\delta_{i}$ respectively.
Since \beqn \int_R^\infty\frac{r}{\log m(B_r)}dr&=&\sum_{k=0}^\infty\int_{R_k}^{R_{k+1}}\frac{r}{\log m(B_r)}dr
 \leq \sum_{k=0}^\infty\left(\frac{R_{k+1}^2-R_k^2}{2\log m(B_{R_k})}\right)\nonumber \\&=&\sum_{k=0}^\infty\frac {3R_k^2}{2\log m(B_{R_k})},
 \eeqn
 we get $\sum_{k=0}^\infty\frac {R_k^2}{\log m(B_{R_k})}=\infty$  by the choice of $R$ (i.e., $R>1$) and (\ref{integral}). On the other hand, we assume that $\lim_{k\rightarrow \infty}m(B_{R_k})=+ \infty$ (otherwise $m(M)<\infty$, which implies the conclusion by Proposition \ref{prop31}).  This means that $\log(T-s)\leq \log m(B_{R_{k+1}})$ as $k$ is large enough. Consequently,
 $$\sum_{k=0}^\infty\delta_k=\frac {\kappa^{*2}}{8\kappa\Lambda^2} \sum_{k=0}^\infty \frac{R_{k+1}^2}{\log(T-s)+\log m(B_{R_{k+1}})}= \sum_{k=0}^\infty\frac {R_k^2}{\log m(B_{R_k})}=\infty.$$ The claim is proved.

  Finally, combining Case 1 with Case 2 yields that $\tau_{k+1}$ can reach $s$ after a finite number of iterations by diminishing some of $\{\delta_k\}$ and hence
 $\lim_{i\rightarrow \infty}P_{s, \tau}^{\nabla u}(f_i)=1$. Since $\tau\in (s, T)$ is chosen arbitrarily, $\lim_{i\rightarrow \infty}P_{s, t}^{\nabla u}(f_i)=1$ for any $t\in (0, T)$.  This finishes the proof. \end{proof}

\begin{proof} [Proof of Theorem \ref{thm31}] If $M$ is complete and $m(M)<\infty$, then the conclusion follows from Proposition \ref{prop31}(4). Now we assume that $M$ is complete and $m(M)=\infty$. In this case, $M$ is noncompact and Ric$_N\geq K$ by the assumption.  We prove the conclusion according to three cases.

 {\it Case 1}.  $ N\in [n, \infty)$ and $K\in \mathbb R$.  By (\ref{V-B-N}), we have
  $\int_1^\infty\frac {r dr}{\log m(B_r^+(x))} =\infty.$  Thus, $\{P_{s, t}^{\nabla u}\}$ is conservative by Proposition \ref{prop33}.

 {\it Case 2.} $N=\infty$ and $K>0$. In this case,  the volume $m(M)$ of $M$ is finite, i.e., $m(M)<\infty$ by Lemma \ref{lem31}. By Proposition \ref{prop31}(4),  the semigroup $\{P_{s, t}^{\nabla u}\}$ is conservative.

{\it Case 3.} $N=\infty$ and $K\leq 0$.  By Lemma \ref{lem31}, we have (\ref{m-Br}) when $K<0$ and (\ref{m-Br-0}) when $K=0$.
It is easy to check that $\int_1^\infty\frac {r dr}{\log m(B_r^+(x))} =\infty$ in both cases. By Proposition \ref{prop33},  $\{P_{s, t}^{\nabla u}\}$ is conservative.
The rest of the proof follows from the same arguments as in the proofs of Proposition \ref{prop31}(4) and Proposition \ref{prop32}.
\end{proof}

\begin{remark}\label{rem32} {\rm In the same way, we can prove that Theorem \ref{thm31} is true for the adjoint heat semigroup $\{\hat P_{t, s}^{\nabla u}\}$.}\end{remark}

\section{Improved Li-Yau's inequality}

 In this section, we will prove Theorem \ref{thm11} based on Theorem \ref{thm31}. For this, we need the following lemmas.

\begin{lem} \label{lem41} Let $(M, F, m)$ be a complete Finsler measure space satisfying (\ref{unif cs cons}) and  $u(t, x)\in H_0^1(M)$ be a positive global solution to $\partial_tu=\Delta u$. Then the following equalities hold pointwise on $M_u$ and on $M$ in the distribution sense.

\noindent (1) $ \partial_t\left(F^2(\nabla u)\right)=2d(\Delta u)(\nabla u)$.

\noindent (2) $ \partial_t\left(F^2(\nabla \log u)\right)=2u^{-1}d(u^{-1}\Delta u)(\nabla u).$

\noindent (3) $\Delta u(t, \cdot)=P_{s,t}^{\nabla u}(\Delta u(s, \cdot))$ for all $0\leq s<t<\infty$.

\noindent (4) $u\Delta \log u=-uF^2(\nabla \log u)+\partial_t u$ for any $t>0$.
\end{lem}
\begin{proof} (1) From Proposition \ref{prop21*},  the solution $u(t, \cdot)$ is smooth on $\cup_{t>0}\left(\{t\}\times M_{u_t}\right)$.  Observe that both sides of the equality are zero a.e. on $M\setminus M_u$ from Lemma \ref{lem21}. On the other hand, (1) holds on $M_u$ by (4.2) in \cite{OS2}. Thus (1) follows.

(2)  Since $d(\log u)=\frac 1u du$, we have $M_{\log u}=M_u$ and $g^{ij}(\nabla \log u )=g^{*ij}(x, d(\log u))=g^{*ij}(x, du)=g^{ij}(\nabla u )$ on $M_{u }$. Consequently, $\nabla (\log u)=\frac 1{u }\nabla u$. By (1), we have on $M_u$
\beqn \partial_t\left(F^2(\nabla \log u)\right)&=&\frac 2{u^{2}}d(\Delta u)(\nabla u)-\frac 2{u} F^2(\nabla \log u)\Delta u \\
&=& 2u^{-1}d(u^{-1}\Delta u)(\nabla u).\eeqn Similarly, both sides of the above equality are zero a.e. on $M\setminus M_u$ by Lemma \ref{lem21}. Thus (2) holds on $M$ in the distribution sense.

(3) By (\ref{unif cs cons}) and Proposition \ref{prop22}, there exists a linearized heat semigroup $\{P_{s, t}^{\nabla u}\}$ for the heat equation.  From (\ref{loc-Lap}), we obtain on $M_u$
 \beqn \partial_s(\Delta u(s, \cdot))&=&\frac 1{\sigma}\frac{\partial}{\partial x^i}\left(\sigma \partial_s\left(g^{ij}(\nabla u(s, \cdot))\right)\frac{\partial u}{\partial x^j}+\sigma g^{ij}(\nabla u(s, \cdot))\frac{\partial^2u}{\partial s\partial x^j}\right)\nonumber \\
 &=&\frac 1{\sigma}\frac{\partial}{\partial x^i}\left(-\sigma g^{ik}(\nabla u(s, \cdot)\partial_s(g_{kl}(\nabla u(s, \cdot))\nabla^{l}u+\sigma g^{ij}(\nabla u(s, \cdot))\frac{\partial^2u}{\partial s\partial x^j}\right)\nonumber \\
 &=& \Delta^{V_s}(\partial_s u(s, \cdot))=\Delta^{V_s}(\Delta u(s, \cdot)),\eeqn where $V_s=\nabla u(s, \cdot)$ and we used $\frac{\partial g_{ij}(x, y)}{\partial y^k}y^k=0$ in the third equality since $g_{ij}(x, y)$ are homogeneous functions of degree zero in $y$. Note that $\Delta u=0$ on $M\setminus M_u$ by Lemma \ref{lem21}.  Thus $\partial_s(\Delta u(s, \cdot))=\Delta^{V_s}(\Delta u(s, \cdot))$ on $M$ in the distribution sense.  For any $\phi\in C_0^\infty(M)$,  $\hat P_{t, s}^{\nabla u}(\phi)\in H_0^1(M)$ can be chosen as a test function. Moreover, $\Delta u\in H_0^1(M)$ by Proposition \ref{prop21*} since $F$ has a finite uniform smoothness constant. It follows from (\ref{AHE}) that
\beqn & &\frac {\partial}{\partial s}\int_M\hat P_{t,s}^{\nabla u}(\phi)\Delta u(s, \cdot)dm \\
 & &=-\int_M\Delta^{V_s}\hat P_{t, s}^{\nabla u}(\phi)\Delta u(s,\cdot)dm+\int_M \hat P_{t, s}^{\nabla u}(\phi)\Delta^{V_s}(\Delta u(s, \cdot))dm \\
& &= \int_M d(\Delta u(s,\cdot))(\nabla^{V_s}\hat P_{t, s}^{\nabla u}(\phi))dm-\int_Md(\hat P_{t, s}^{\nabla u}(\phi))(\nabla^{V_s}(\Delta u(s,\cdot)))dm=0,\eeqn which implies
\beqn \int_M\phi\Delta u(t, \cdot)dm=\int_M\hat P_{t,s}^{\nabla u}(\phi)\Delta u(s, \cdot)dm=\int_M\phi P_{s,t}^{\nabla u}(\Delta u(s, \cdot))dm.\eeqn Thus (3) follows.

(4) Since  $g^{ij}(\nabla \log u )=g^{ij}(\nabla u )$ and $\nabla (\log u)=\frac 1{u }\nabla u$ on $M_{u }=M_{\log u}$, we have
\beq  \Delta \log u =-F^2(\nabla \log u)+u^{-1}\Delta u\label{log-u} \eeq  on $M_u$ by (\ref{loc-Lap}). Moreover, this equality holds a.e. on $M$ since $\Delta \log u=\Delta u=0$ a.e. on $M\setminus M_u$. Thus (4) holds on $M$ in the distribution sense.
\end{proof}

For the sake of simplicity, we write $u_t$ instead of $u(t, \cdot)$ and $\partial_tu_t$ instead of $\partial_t u(t, \cdot)$ in the following to emphasize the dependence on the time $t$. If $(M, F, m)$ is a compact Finsler measure space (without boundary), then (\ref{unif cs cons}) is trivial,   $C_0^\infty(M)=C^\infty(M)$ and $H_{loc}^1(M)=H_0^1(M)=H^1(M)=H_c^1(M)$.
The following lemma plays an important role in the subsequent arguments.
\begin{lem} \label{lem42*} Let $(M, F, m)$ be a compact Finsler measure space and $u_t\in H^1(M)$ be a positive global solution to $\partial_tu=\Delta u$ (in the weak sense). Assume that Ric$_N\geq K$ for some $N\in [n, \infty]$ and $K\in \mathbb R$. For any $\sigma\in [0, t]$ and nonnegative $\phi\in C^\infty(M)$ (or $H^1(M)$), define
\beq H(\sigma): =\int_M \phi P_{\sigma, t}^{\nabla u}(u_\sigma\log u_\sigma)dm. \label{H-sigma} \eeq Then
\beq H'(\sigma)=-\int_M\phi P_{\sigma, t}^{\nabla u}\big(u_\sigma F^2(\nabla \log u_\sigma)\big)dm,\label{H-der-1}\eeq
and \beq H''(\sigma) \geq -2KH'(\sigma)+\frac 2N\int_M \hat P_{t, \sigma}^{\nabla u}(\phi) u_\sigma\big(\Delta\log u(\sigma, \cdot)\big)^2 dm,\label{H-2der}\eeq where the RHS of (\ref{H-2der}) should be understood as a limit when $N=\infty$.
\end{lem}
\begin{proof} Let $f:=u_0\in H^1(M)$. Then $u_t=P_{0, t}^{\nabla u}(f)$  by Proposition \ref{prop31}(1). Further $u_t\in H^2(M)\cap C^{1, \beta}(M)$ with $\Delta u_t\in H^1(M)$  and $F(\nabla u_t)\in H^1(M)\cap C^\beta(M)$ by Proposition \ref{prop21*}. Hence $u_t$ and $F(\nabla u_t)$ are bounded on $M$ because of the H\"older continuity, which means that $F^2(\nabla u_t)\in H^1(M)\cap C^\beta(M)$.  Since  \beq \nabla (\log u_t)=\frac{\nabla u_t}{u_t}, \ \ \ \Delta (\log u_t)=\frac{\Delta u_t}{u_t}- \frac{F^2(\nabla u_t)}{u_t^2} \label{uuuu} \eeq pointwise on $M_{u_t}=M_{\log u_t}$ and on $M$ in the distribution sense by Lemma \ref{lem41},  the functions $u_t\log u_t$, $u_tF^2(\nabla\log u_t)$, $\Delta \log u_t$  are in $H^1(M)$ and $u_t\big(\Delta \log u_t)^2\in L^1(M)$. Hence the integrals of the RHS in (\ref{H-sigma})--(\ref{H-2der}) make sense.

 Let $\psi$ be a smooth function in one variable defined by $\psi(\upsilon)=\upsilon\log \upsilon$ for any $\upsilon\in (0, \infty)$. Then  $\psi''(\upsilon)=\frac 1{\upsilon}$ and
$$H(\sigma)=\int_M \phi P_{\sigma, t}^{\nabla u}(\psi(u_\sigma))dm=\int_M \psi(u_\sigma)\hat P_{t, \sigma}^{\nabla u}(\phi)dm$$ for any $\sigma\in [0, t]$. The derivative of $H$ with respect to $\sigma$ is given by
\beq H'(\sigma)&=&\int_M\psi'(u_\sigma)(\partial_\sigma u_{\sigma})\hat P_{t, \sigma}^{\nabla u}(\phi)dm+\int_M \psi(u_\sigma)\partial_{\sigma}(\hat P_{t, \sigma}^{\nabla u}(\phi))dm\nonumber \\
 &=&-\int_M d\left(\psi'(u_\sigma) \hat P_{t, \sigma}^{\nabla u}(\phi)\right)(\nabla u_\sigma) dm+\int_M \psi'(u_\sigma)du_\sigma(\nabla^{V_\sigma}(\hat P_{t, \sigma}^{\nabla u}(\phi)))dm \nonumber \\
&=& -\int_M \psi''(u_\sigma)F^2(\nabla u_{\sigma}) \hat P_{t, \sigma}^{\nabla u}(\phi)dm=-\int_M u_\sigma F^2(\nabla \log u_\sigma)\hat P_{t, \sigma}^{\nabla u}(\phi)dm\nonumber \\
&=& -\int_M \phi P_{\sigma, t}^{\nabla u}\big(u_\sigma F^2(\nabla \log u_\sigma)\big)dm,\label{H-der-1*} \eeq where we used $d(\hat P_{t, \sigma}(\phi)(\nabla u_\sigma)=du_\sigma(\nabla^{V_\sigma}(\hat P^{\nabla u}_{t, \sigma}(\phi))$ in the third equality.  Let
 \beq G(\sigma):=\int_M \phi P_{\sigma, t}^{\nabla u}\big(u_\sigma F^2(\nabla \log u_\sigma)\big)dm=\int_M u_\sigma F^2(\nabla\log u_\sigma)\hat P_{t, \sigma}^{\nabla u}(\phi)dm.\label{G-def}\eeq Then $H'(\sigma)=-G(\sigma)$. From Lemma \ref{lem41}(2) and (\ref{AHE}), we have
\beqn G'(\sigma)&=& -\int_M d\left(F^2(\nabla \log u_\sigma) \hat P_{t, \sigma}^{\nabla u}(\phi)\right)(\nabla u_\sigma) dm +2\int_Md\left(u_\sigma^{-1}\Delta u_\sigma\right)(\nabla u_\sigma)\hat P_{t, \sigma}^{\nabla u}(\phi)dm \nonumber \\
& & +\int_M d\left(u_\sigma F^2(\nabla \log u_\sigma)\right)\left(\nabla^{V_\sigma} \hat P_{t, \sigma}^{\nabla u}(\phi)\right) dm \nonumber \\
&=&-\int_M d\left(F^2(\nabla \log u_\sigma) \right)(\nabla u_\sigma)\hat P_{t,\sigma}^{\nabla u}(\phi) dm+2\int_Md\left(u_\sigma^{-1}\Delta u_\sigma\right)(\nabla u_\sigma)\hat P_{t, \sigma}^{\nabla u}(\phi)dm \nonumber \\
& &+\int_M u_\sigma d\left( F^2(\nabla \log u_\sigma)\right)(\nabla^{V_\sigma} \hat P_{t, \sigma}^{\nabla u}(\phi)) dm.\label{H-der-2}\eeqn
 By (\ref{log-u}), we have
\beqn d\left(\Delta\log u_\sigma\right)(\nabla u_\sigma) =d\left(-F^2(\nabla \log u_{\sigma})+u_\sigma^{-1}\Delta u_\sigma\right)(\nabla u_\sigma)\label{uu} \eeqn on $M_{u_\sigma}$. Further, this equality holds on $M$ in the distribution sense.  Hence,
 \beq G'(\sigma)&=& \int_M  d \left(u_\sigma\hat P_{t, \sigma}^{\nabla u}(\phi)\right)\left(\nabla^{V_\sigma}F^2(\nabla \log u_\sigma)\right) dm
 \nonumber \\& &+ 2\int_M u_\sigma \hat P_{t, \sigma}^{\nabla u}(\phi) d\left(\Delta\log u_\sigma\right)(\nabla \log u_\sigma)dm.
\label{G-d}  \eeq
Since $\phi\in C^\infty(M)\subset H^1(M)\cap L^\infty(M)$ and $u_\sigma$ is bounded on $M$, $u_\sigma \hat P_{t, \sigma}^{\nabla u}(\phi)\in H^1(M)\cap L^\infty(M)$ by Proposition \ref{prop22} and Corollary \ref{cor31}. Moreover, $\log u_\sigma\in H^2(M)\cap C^1(M)\cap H^1(M)$ and $\Delta \log u_\sigma\in H^1(M)$ by previous arguments. Applying (\ref{Bochner4''}) in (\ref{G-d}) with the test function $u_\sigma \hat P_{t, \sigma}^{\nabla u}(\phi)$ yields
\beq H''(\sigma)&=& -G'(\sigma)\geq  2\int_M u_\sigma\hat P_{t, \sigma}^{\nabla u}(\phi)\left\{KF^2(\nabla \log u_\sigma)+\frac {(\Delta\log u_\sigma)^2}N\right\}dm\nonumber \\
&=& -2KH'(\sigma)+\frac 2N\int_M \hat P_{t, \sigma}^{\nabla u}(\phi) \left(u_\sigma(\Delta\log u_\sigma)^2\right)dm \label{HG-d}\eeq for $N\in [n, \infty)$. Thus (\ref{H-2der}) follows.
If $N=\infty$, applying (\ref{Bochner3''}) in (\ref{G-d}) with the test function $u_\sigma \hat P_{t, \sigma}^{\nabla u}(\phi)$ gives $H''(\sigma)\geq -2KH'(\sigma)$. Combining these two cases yields the conclusion. \end{proof}

Based on Lemma \ref{lem41}, we prove Theorem \ref{thm11}.
\begin{proof}[Proof of Theorem \ref{thm11}]
 We keep the notations in the proof of Lemma \ref{lem42*}. Let $G(\sigma)$ be defined by (\ref{G-def}) for any $\sigma\in [0, t]$. For any nonnegative $\phi\in C^\infty(M)$, by (\ref{H-2der}), we have
\beqn -G'(\sigma)\geq 2KG(\sigma)+\frac 2N\int_{M}\hat P_{t, \sigma}^{\nabla u}(\phi) \big(u_\sigma (\Delta \log u_\sigma)^2\big) dm.\label{GG-d-1}\eeqn
  For any continuous function $\mu=\mu(\sigma)$ on $(0, \infty)$, we have $(\Delta \log u_\sigma)^2\geq 2\mu \Delta\log u_\sigma-\mu^2$ in the distribution sense. Consequently, by Lemma \ref{lem41},
 \beqn -G'(\sigma) &\geq & 2KG(\sigma)+\frac {4\mu}N\int_{M}\hat P_{t, \sigma}^{\nabla u}(\phi)u_\sigma\Delta(\log u_\sigma)dm-\frac{2\mu^2}N\int_M \phi P_{\sigma, t}^{\nabla u} (u_\sigma) dm\\
 &=& 2KG(\sigma)+\frac {4\mu}N\int_{M}\hat P_{t, \sigma}^{\nabla u}(\phi)\left(-u_\sigma F^2(\nabla \log u_\sigma)+\Delta u_\sigma\right)dm-\frac{2\mu^2}N\int_M \phi u_t dm\\
 &=&\left(2K-\frac{4\mu}N\right)G(\sigma)+\frac {4\mu}N\int_M\phi\Delta u_tdm-\frac{2\mu^2}N\int_M\phi u_tdm.
 \label{GG-d-2}\eeqn
 For any $s\in [0, t]$, consider the function $$\mathcal G(s):=a(t-s)G(t-s).$$ Let $\sigma: =t-s$. Choose $\mu=\mu(\sigma)$ such that $(\log a)'(\sigma)-2K+4\mu/N=0$, namely, $\mu=\frac N4\left(2K-(\log a)'(\sigma)\right)$. Then
 \beqn \mathcal G'(s) &=& -a'(\sigma)G(\sigma)-a(\sigma)G'(\sigma)\\
 &\geq & a(\sigma)\left(-(\log a)'(\sigma)+2K\right)\int_M\phi \Delta u_tdm-\frac{N}8a(\sigma)\left(-(\log a)'(\sigma)+2K\right)^2\int_M\phi u_tdm. \eeqn
Integrating this inequality in $s$ from $0$ to $t$ and using $\lim\limits_{t\rightarrow 0+}a(t)=0$ yield
\beqn \int_M\phi u_tF^2(\nabla \log u_t)dm &\leq & \left(1-\frac{2K}{a(t)}\int_0^t a(s)ds\right)\int_M\phi \partial_t u_tdm  \nonumber \\
&+&\left\{\frac N{8a(t)}\int_0^t\frac{a'(s)^2}{a(s)}ds-\frac{NK}2+\frac{NK^2}{2a(t)}\int_0^ta(s)ds\right\}\int_M\phi  u_tdm.
 \eeqn
Since  $\phi$ is arbitrary, the above inequality implies that $u_tF^2(\nabla \log u_t)\leq \alpha(t) \partial_t u_t+\varphi(t)u_t$ a.e. on $M$, i.e.,
\beqn F^2(\nabla \log u_t) -\alpha(t) \partial_t (\log u_t)\leq \varphi(t)  \eeqn a.e. on $M$. Since $F(\nabla \log u_t)$ and $\partial_t (\log u_t)$ are continuous on $M$, this inequality holds on $M$.
 This finishes the proof. \end{proof}

For the complete and noncompact case, Lemma \ref{lem42*} still holds under the assumptions in Theorem \ref{thm41}. The proof indeed follows the same lines as  those of Lemma \ref{lem42*}. The key point is to check that all integrals in (\ref{H-sigma})--(\ref{H-2der}) are well defined and (\ref{Bochner3''})--(\ref{Bochner4''}) can be applied  under the assumptions on $f$ and $u$ in Theorem \ref{thm41}.  More precisely, we have
\begin{lem} \label{lem43} Let $(M, F, m)$  and $u(t, \cdot)$ be as in Theorem \ref{thm41}. Then the functions $u\log u$, $uF^2(\nabla\log u)$, $\Delta\log u$ are in $H_0^1(M)$ and $u(\Delta \log u)^2\in L^1(M)$.
Further, assume that Ric$_N\geq K$ for some $N\in [n, \infty]$ and $K\in \mathbb R$. If $H(\sigma)$ is given by (\ref{H-sigma}) for any $\sigma\in [0, t]$ and $\phi\in C_0^\infty(M)$, then we have (\ref{H-2der}), whose right--hand side should be understood as a limit when $N=\infty$.
\end{lem}

\begin{proof}  By Proposition \ref{prop22}, (\ref{unif cs cons}) ensures the existence of the linearized heat semigroup $\{P_{s, t}^{\nabla u}\}$. Since $(M, F)$ is complete and $F$ has finite reversibility by (\ref{unif cs cons}), we have $H_0^1(M)=H^1(M)$ by Lemma 11.4 in \cite{Oh1}. Hence it suffices to consider the Sobolev space $H^1(M)$. We still denote $u(t, \cdot)$ by $u_t$ as before.

 Let  $f: =u_0\in H^1(M)\cap L^\infty(M)$. Then, for any $t>0$, $u_t=P_{0, t}^{\nabla u}(f)\in H^1(M)\cap L^\infty(M)$ by Proposition \ref{prop31}(1) and Theorem \ref{thm31}. Hence $u_t\log u_t\in H^1(M)$. Moreover, by Proposition \ref{prop21*},  $u_t\in C^{1, \beta}(M)\cap H_{loc}^2(M)$ with $\Delta u_t\in H^1(M)$ since $F$ has a finite uniform smoothness constant. From this,  Lemma \ref{lem41}(4) and $F^2(\nabla u_t)\in H^1(M)$ by the assumption,  the functions $u_tF^2(\nabla\log u_t)$ and $\Delta(\log u_t)$ are in $H^1(M)$. Further, $u_t(\Delta \log u_t)^2\in L^1(M)$.
 Consequently,  all integrals of the RHS in  (\ref{H-sigma})--(\ref{H-2der}) are well defined.

By the same arguments as in the proof of Lemma \ref{lem42*}, we get (\ref{G-d}). Since $\phi\in C_0^\infty(M)\subset H_0^1(M)\cap L^\infty(M)$ (i.e., $H^1(M)\cap L^\infty(M)$),  $\hat P_{t, \sigma}^{\nabla u}(\phi)\in H^1(M)\cap L^\infty(M)$ by Proposition \ref{prop22} and Theorem \ref{thm31}. Thus $u_\sigma\hat P_{t, \sigma}^{\nabla u}(\phi)\in H^1(M)\cap L^\infty(M)$ for any $\sigma\in [0, t]$. Moreover, since $F^2(\nabla u_t)\in H^1(M)$ by the assumption, we get $F(\nabla(F^2(\nabla u_t)))\in L^2(M)$. Note that (\ref{unif cs cons}) (equivalently,  (\ref{unif cons})) implies that any two weighted Riemannian metrics are comparable and each of them is comparable to $F$. Therefore, $F(\nabla^{\nabla u}(F^2(\nabla u_t)))\in L^2(M)$.
 By previous arguments,  $\log u_t\in H^2_{loc}(M)\cap C^1(M)\cap H^1(M)$ and $\Delta \log u_t\in H^1(M)$.  From these and Remark \ref{rem21},  applying (\ref{Bochner3''}) and (\ref{Bochner4''}) respectively in (\ref{G-d}) with the test function $u_\sigma \hat P_{t, \sigma}^{\nabla u}(\phi)$ yields (\ref{H-2der}) in the cases when $N\in [n, \infty)$ and $N=\infty$, where the RHS of (\ref{H-2der}) is understood as a limit when $N=\infty$.\end{proof}

\begin{proof}[Proof of Theorem \ref{thm41}]  The proof follows from that of Theorem \ref{thm11} provided that we use Lemma \ref{lem43} instead of Lemma \ref{lem42*}. We omit it here.  \end{proof}

\section{Generalized Li-Yau's inequality}

The following lemma is elementary, whose proof can be found in \cite{Vi} (see Theorem 14.28, \cite{Vi}).
\begin{lem} \label{lem51} Let $\lambda\in \mathbb R$ and $f$ be a nonnegative $C^2$ function on $[0, t]$ such that $f''(s)+\lambda f(s)\leq 0$ on $[0, t]$.

\noindent (1) If $\lambda>\frac {\pi^2}{t^2}$, then $f(s)=0$ for all $s\in [0, t]$.

\noindent (2) If $\lambda=\frac{\pi^2}{t^2}$, then $f(s)=c\sin\left(\frac st\pi\right)$ for some $c\geq 0$.

\noindent (3) If $\lambda<\frac {\pi^2}{t^2}$, then $f(s)\geq \tau_{\lambda}(t-s)f(0)+\tau_{\lambda}(s)f(t)$, where
\begin{eqnarray*}\tau_{\lambda}(s)=\left\{\begin{array}{lll}\frac {\sin(s\sqrt{\lambda})}{\sin(t\sqrt{\lambda})}, \ & \lambda>0, \\
\frac st, \ & \lambda=0, \\
\frac {\sinh(s\sqrt{-\lambda})}{\sinh(t\sqrt{-\lambda})}, \ & \lambda<0. \end{array}\right. \label{tau} \end{eqnarray*}
 \end{lem}

Let $u_t$ be a positive global solution to $\partial_t u=\Delta u$ on a Finsler measure space $(M, F, m)$ satisfying (\ref{unif cs cons}). For any nonnegative $\phi\in C_0^\infty(M)$ such that $\phi$ is not identically zero on $M$ and $K\in \mathbb R\setminus\{0\}$, let
\beq \zeta(t): =\frac 2{N\int_M\phi u_tdm}, \ \ \ \ \ \ \  \chi(t):=\frac {4}{NK}\frac{\int_M\phi\partial_t u_tdm}{\int_M\phi u_tdm},\label{zeta-chi}\eeq  which are well defined by Proposition \ref{prop21*}.  Further, define the functions $\Psi_t(x)$ and $\tilde\Psi_t(x)$ on $(-\infty, 1+\frac{\pi^2}{K^2t^2})$ by
\beq \Psi_t(x):=\left\{\begin{array}{ll}\frac K2\left(x-2+2\sqrt{x-1}\cot(Kt\sqrt{x-1})\right), & 1< x<1+\frac{\pi^2}{K^2t^2}, \\
-\frac K2+\frac 1t, & x=1, \\
\frac K2\left(x-2+2\sqrt{1-x}\coth(Kt\sqrt{1-x})\right), & x< 1, \end{array}\right. \label{Psi-t}\eeq
and $\tilde\Psi_t(x):=\Psi_t(x)-Kx+2K$. Then $\Psi_t(x)$ is smooth and strictly concave on $(-\infty, 1+\pi^2/(K^2t^2))$ (cf. Lemma 2.8, \cite{BBG}). We first consider the compact case.

\begin{prop} \label{prop51}  Let $(M, F, m)$ be an $n$-dimensional compact Finsler measure space and $u_t$ be the positive global solution to $\partial_t u=\Delta u$ (in the weak sense) on $[0, \infty)$ with the initial $u_0:=f\in H^1(M)$. Assume that Ric$_N\geq K$ for some $N\in [n, \infty)$ and $ K\in \mathbb R\setminus\{0\}$. For any nonnegative $\phi\in C^\infty(M)$ such that $\phi$ is not identically zero on $M$, let $\zeta(t), \chi(t), \Psi_t$ and $\tilde \Psi_t$ be the functions defined by (\ref{zeta-chi}) and (\ref{Psi-t}) respectively. Then $\chi(t) < 1+\frac{\pi^2}{K^2t^2}$.
Further, we have
\beq & & \exp\left\{\zeta\int_M\phi\left(u_t\log u_t-P_{0, t}^{\nabla u}(f\log f)\right)dm+\frac {Kt}2\chi-Kt\right\}\nonumber \\
& &\leq \left\{\begin{array}{lll} \frac{\sin \left(Kt\sqrt{\chi-1}\right)}{K\sqrt{\chi-1}}\Big\{-\zeta\int_M\phi u_tF^2(\nabla\log u_t) dm +\Psi_t(\chi)\Big\}, & 1<\chi<1+\frac {\pi^2}{K^2t^2}, \\
t\Big\{-\zeta\int_M\phi u_tF^2(\nabla \log u_t)dm-\frac K2+\frac 1t\Big\}, & \chi=1,\\
\frac{\sinh \left(Kt\sqrt{1-\chi}\right)}{K\sqrt{1-\chi}}\Big\{-\zeta\int_M\phi u_tF^2(\nabla\log u_t) dm +\Psi_t(\chi)\Big\}, & \chi< 1.
\end{array}\right. \label{log-sob-ineq-1}\eeq
and
  \beq & & \exp\left\{-\zeta\int_M\phi\left(u_t\log u_t-P_{0, t}^{\nabla u}(f\log f)\right)dm -\frac {Kt}2\chi+Kt\right\}\nonumber \\
& &\leq \left\{\begin{array}{lll} \frac{\sin \left(Kt\sqrt{\chi-1}\right)}{K\sqrt{\chi-1}}\left\{\zeta\int_M fF^2(\nabla \log f)\hat P_{t, 0}^{\nabla u}(\phi)dm +\tilde \Psi_t(\chi)\right\}, & 1<\chi<1+\frac {\pi^2}{K^2t^2}, \\
t\Big\{-\zeta\int_M fF^2(\nabla \log f)\hat P_{t, 0}^{\nabla u}(\phi)dm+\frac K2+\frac 1t\Big\}, & \chi = 1,\\
\frac{\sinh \left(Kt\sqrt{1-\chi}\right)}{K\sqrt{1-\chi}}\left\{\zeta\int_M fF^2(\nabla \log f)\hat P_{t, 0}^{\nabla u}(\phi) dm +\tilde \Psi_t(\chi)\right\}, & \chi>1,\end{array}\right. \label{log-sob-ineq-2}\eeq where $\zeta=\zeta(t)$ and $\chi=\chi(t)$.
 \end{prop}
\begin{proof}  Since $M$ is compact,  we have $C_0^\infty(M)=C^\infty(M)$ and $H_0^1(M)=H^1(M)$. Moreover, $u_t=P_{0, t}^{\nabla u}(f)\in H^1(M)$ for any $t>0$ as before. For any $\sigma\in [0, t]$ and nonnegative $\phi\in C^\infty(M)\subset H^1(M)\cap L^\infty(M)$ such that $\phi$ is not identically zero on $M$, consider the function
\beqn H(\sigma)=\int_M \phi P_{\sigma, t}^{\nabla u}(u_{\sigma}\log u_{\sigma})dm  \label{H-def} \eeqn  as in Lemma \ref{lem42*}.
 From Lemma \ref{lem42*}, we have (\ref{H-2der}). Next we estimate the second term on the RHS of (\ref{H-2der}) in a different way from that in the proof of Theorem \ref{thm11}.  Note that $0\leq \hat P_{t, \sigma}^{\nabla u}(\phi)\in H^1(M)\cap L^\infty(M)$ by Corollary \ref{cor31}.  By Lemma \ref{lem41} and Cauchy-Schwarz's inequality, we have
\beq & &\left(\int_M\hat P_{t, \sigma}^{\nabla u}(\phi)\big(\Delta u_\sigma-u_\sigma F^2(\nabla \log u_{\sigma})\big)dm\right)^2\nonumber \\
& &=\left(\int_M \big(u_\sigma\Delta\log u_{\sigma}\big)\hat P_{t, \sigma}^{\nabla u}(\phi)dm\right)^2\nonumber \\
& &\leq \int_M u_\sigma\hat P_{t, \sigma}^{\nabla u}(\phi) dm\cdot \int_M\hat P_{t, \sigma}^{\nabla u}(\phi) \left(u_\sigma(\Delta \log u_\sigma)^2\right)dm. \label{ppp}\eeq
Since $$\int_M u_\sigma\hat P_{t, \sigma}^{\nabla u}(\phi) dm=\int_M \phi  P_{\sigma, t}^{\nabla u}(u_\sigma)dm=\int_M \phi u_tdm$$ by Proposition \ref{prop31},  it follows from (\ref{ppp}) and Lemma \ref{lem41} that
\beq  \int_M\hat P_{t, \sigma}^{\nabla u}(\phi) \left(u_\sigma(\Delta \log u_\sigma)^2\right)dm &\geq &\frac N2\zeta(t) \left(\int_M\hat P_{t, \sigma}^{\nabla u}(\phi)\big(\Delta u_\sigma-u_\sigma F^2(\nabla \log u_{\sigma})\big)dm\right)^2\nonumber \\
&=& \frac N2\zeta(t) \left(\int_M\big(\phi P_{\sigma, t}^{\nabla u}(\Delta u_\sigma)+H'(\sigma)\big)dm\right)^2\nonumber \\
&=& \frac N2\zeta(t) \left(\int_M\big(\phi \Delta u_t+H'(\sigma)\big)dm\right)^2.\label{ppp*}\eeq
Let $\sigma:=t-s$ for $s\in [0, t]$. From  (\ref{H-2der}), (\ref{ppp*}) and  $\partial_s(H(t-s))=-H'(\sigma)$, we have
\beq \partial_s^2\left(H(t-s)\right)\geq 2K\partial_s(H(t-s))+\zeta(t)\left(\int_M \big\{\phi\Delta u_tdm-\partial_s(H(t-s))\big\}dm\right)^2.\label{HH-eq}\eeq
 Set \beqn & &\vartheta(t) :=\int_M\phi\Delta u_tdm-\frac{K}{\zeta(t)}=\int_M\phi \partial_tu_tdm -\frac{K}{\zeta(t)}, \\
   & &\nu(t) :=\frac {2\vartheta K}{\zeta(t)}+\frac{K^2}{\zeta^2(t)}. \eeqn  In the following, we simply write $\zeta, \vartheta$ and $\nu$ instead of $\zeta(t)$, $\vartheta(t)$ and $\nu(t)$ respectively. Thus (\ref{HH-eq}) can be rewritten as
\beqn \partial_s^2(H(t-s)) \geq  \zeta\left\{ \left(\partial_s(H(t-s))-\vartheta\right)^2 +\nu \right\}.\label{HH-eq*}
\eeqn
Let $h(s):=\exp\left\{-\zeta\left(H(t-s)-\vartheta s\right)\right\}$ for any $s\in [0, t]$. The above inequality is reduced to
\beq h''(s)+\zeta^2\nu\  h(s)\leq 0.\label{h-der-2}\eeq
Since $h>0$, both cases (1) and (2) in Lemma \ref{lem51} cannot happen. This means that the case (3) in Lemma \ref{lem51} must occur, i.e.,  $$\zeta^2\nu=K^2(\chi-1)<\frac{\pi^2}{t^2},$$ equivalently, $\chi<1+\pi^2/(t^2 K^2)$. In this case, we have
\beq h(s)\geq \tau_{\zeta^2\nu}(t-s)h(0)+\tau_{\zeta^2\nu}(s) h(t).\label{ff}\eeq Observe that the above inequality becomes equality when $s=0$ since $\tau_{\zeta^2\nu}(t-s)|_{s=0}=1$ and $\tau_{\zeta^2\nu}(0)=0$. Let
$$\bar h(s):=h(s)-\tau_{\zeta^2\nu}(t-s)h(0)-\tau_{\zeta^2\nu}(s)h(t).$$ Then $\bar h(s)\geq 0$ and $\bar h(0)=0$. This means $\bar h'(0)\geq 0$. Consequently,
\beq h'(0)\geq -\tau_{\zeta^2\nu}'(t)h(0)+\tau_{\zeta^2\nu}'(0)h(t). \label{fff}\eeq
 Note that $h(t)=\exp\left\{-\zeta\int_M\phi P_{0,t}^{\nabla u}(f\log f)dm+\zeta\vartheta t\right\}$ and
 $$h(0)=\exp\left\{-\zeta\int_M\phi u_t\log u_tdm\right\}, \ \ h'(0)=-\zeta\left(\int_M\phi u_tF^2(\nabla \log u_t)dm-\vartheta\right)h(0).$$

{\it Case 1}.  When $1<\chi<1+\pi^2/(t^2 K^2)$, i.e., $\zeta^2\nu>0$,  by (\ref{fff}), we have
 \beq & & \exp\left\{\zeta\int_M\phi\left(u_t\log u_t-P_{0, t}^{\nabla u}(f\log f)\right)dm+\frac {Kt}2\chi-Kt\right\}\nonumber \\
 & & \leq \frac{\sin \left(Kt\sqrt{\chi-1}\right)}{K\sqrt{\chi-1}}\left\{-\zeta\int_M\phi u_tF^2(\nabla\log u_t) dm +\Psi_t(\chi)\right\}. \label{exp-1}\eeq

{\it Case 2}. When $\chi<1$, i.e., $\zeta^2\nu<0$, by (\ref{fff}), the left side of (\ref{exp-1}) is not greater than
 \beq  \frac{\sinh \left(Kt\sqrt{1-\chi}\right)}{K\sqrt{1-\chi}}\left\{-\zeta\int_M\phi u_tF^2(\nabla\log u_t) dm +\Psi_t(\chi)\right\}. \label{exp-2}\eeq

{\it Case 3}.  When $\chi=1$, i.e., $\zeta^2\nu=0$,  by (\ref{fff}), the left side of (\ref{exp-1}) is not greater than
 \beqn  -t\zeta\int_M\phi u_tF^2(\nabla \log u_t)dm-\frac{Kt}2+1, \eeqn
which is the limit of the RHS of (\ref{exp-1}) or (\ref{exp-2}) as $\chi\rightarrow 1$ since $$\lim\limits_{\chi\rightarrow 1}\frac {\sin(Kt\sqrt{\chi-1})}{Kt\sqrt{\chi-1}}=\lim\limits_{\chi\rightarrow 1}\frac {\sinh(Kt\sqrt{1-\chi})}{Kt\sqrt{1-\chi}}=1.$$ Combining the above three cases together gives (\ref{log-sob-ineq-1}).
Similarly, since $\bar h(s)\geq 0$ by (\ref{ff}) for $s\in [0, t]$ and $\bar h(t)=0$, we have $\bar h'(t)\leq 0$, i.e.,
\beq h'(t)\leq -\tau_{\zeta^2\nu}'(0)h(0)+\tau_{\zeta^2\nu}'(t)h(t),\label{fff*} \eeq which implies (\ref{log-sob-ineq-2}).
 \end{proof}

For the case when Ric$_N\geq 0$, i.e., $K=0$, we have the following result.
\begin{prop} \label{prop52} Let $(M, F, m)$ and $u_t$ be the same as in Proposition \ref{prop51}. Assume that Ric$_N\geq 0$ for some $N\in[n, \infty)$. Then, for any $0\leq s<t<\infty$ and $0\leq \phi\in C^\infty(M)$ such that $\phi$ is not identically zero, the following inequalities hold.
\beq & &\exp\left(\zeta\int_M\phi\left\{u_t\log u_t-P_{s, t}^{\nabla u}(u_s\log u_s)+(t-s)\Delta u_t\right\}dm\right)\nonumber \\
& & \leq 1+(t-s)\zeta\int_M\phi u_t\Delta(\log u_t)dm, \label{exp-uu1} \eeq
\beq & & \exp\left(-\zeta\int_M\phi\left\{u_t\log u_t-P_{s, t}^{\nabla u}(u_s\log u_s)+(t-s)\Delta u_t\right\}dm\right)\nonumber \\
& & \leq 1-(t-s)\zeta\left(\int_M  (u_s\Delta \log u_s)\hat P_{t, s}^{\nabla u}(\phi)\right)dm, \label{exp-uu2}\eeq
and
 \beq  \int_M \phi u_t\Delta \log u_t dm \geq \int_M (u_s\Delta \log u_s)\hat P_{t, s}^{\nabla u}(\phi)dm\left\{1+ \zeta(t-s)\int_M\phi u_t\Delta \log u_t dm\right\}, \label{exp-uu3}\eeq where $\zeta=\zeta(t)$ defined by (\ref{zeta-chi})$_1$. \end{prop}
\begin{proof}   From the proof of Proposition \ref{prop51}, we have (\ref{h-der-2}). Since $K=0$, (\ref{h-der-2}) is reduced to $h''(s)\leq 0$. Hence, for any $s\in[0, t)$ and $\tau\in [0, t-s]$, we have $h'(\tau)\leq h'(0)$. Integrating this with respect to $\tau$ from $0$ to $t-s$ yields
\beq h(t-s)\leq h(0)+h'(0)(t-s),\label{hhh-1}\eeq  which implies (\ref{exp-uu1}).
Similarly, since $h''(s)\leq 0$ for any $s\in[0, t)$, we have $h'(\tau)\geq h'(t-s)$ for any $\tau\in [0, t-s]$. Integrating this with respect to $\tau$ from $0$ to $t-s$ yields
\beq h(0)\leq h(t-s)-(t-s)h'(t-s),\label{hhh-2} \eeq  which implies (\ref{exp-uu2}).
Combining (\ref{exp-uu1}) with (\ref{exp-uu2}) together yields (\ref{exp-uu3}).
\end{proof}

\begin{remark} \label{rem51} {\rm (1) For the sake of simplicity, we take $s=0$ in Proposition \ref{prop51}. In fact, Proposition \ref{prop51} is true for any $0\leq s<t<\infty$, in which we use $P_{s, t}^{\nabla u}$ and $t-s$ instead of $P_{0, t}^{\nabla u}$ and $t$ respectively as in Proposition \ref{prop52}.

(2) It is easy to see that $\Psi_t(\chi)$ and the right-hand sides of (\ref{log-sob-ineq-1})-(\ref{log-sob-ineq-2}) are continuous at $\chi=1$. Moreover, (\ref{exp-uu1}) and (\ref{exp-uu2}) with $s=0$ are just the limits of (\ref{log-sob-ineq-1}) and (\ref{log-sob-ineq-2}) respectively as $K$ goes to zero.
}\end{remark}

Based on Propositions \ref{prop51}-\ref{prop52}, we obtain the following generalized Li-Yau's inequality.

\begin{thm} \label{thm51}  Let $(M, F, m)$ be an $n$-dimensional compact Finsler measure space and $u(t, \cdot)\in H^1(M)$ be a positive global solution to $\partial_t u=\Delta u$ (in the weak sense) on $[0, \infty)$. Assume that Ric$_N\geq K$ for some $N\in [n, \infty)$ and $ K\in \mathbb R$. For any nonnegative $\phi\in C^\infty(M)$ such that $\phi$ is not identically zero on $M$, let $\chi(t)$ and $\Psi_t$ be the functions defined by (\ref{zeta-chi})$_2$ and (\ref{Psi-t}) respectively.

(1) If $K\neq 0$, then $\chi(t)<1+\frac{\pi^2}{K^2t^2}$ and \beq \frac 4{NK}\partial_t(\log u)<1+\frac {\pi^2}{K^2t^2}\label{u-t-der} \eeq on $M$ for all $t>0$. Moreover,
\beq \int_M \phi u F^2(\nabla\log u)dm\leq \frac N2\Psi_t(\chi) \int_M \phi udm. \label{imp-Li-Yau-ineq}  \eeq

(2) If $K=0$, it holds that
\beq F^2(\nabla(\log u))-\partial_t(\log u) \leq \frac N{2t}\label{LY-0}\eeq for all $t>0$.
\end{thm}

\begin{proof} When $K\neq 0$, the upper bound estimate of $\chi(t)$ follows from Proposition \ref{prop51}. We write $u_t$ and $\partial_tu_t$ instead of $u(t, \cdot)$ and $\partial_t u(t, \cdot)$ as before. By definition, we have
$$\frac 4{NK}\int_M\phi\partial_t u_tdm=\chi(t)\int_M\phi u_tdm<\left(1+\frac {\pi^2}{K^2t^2}\right)\int_M\phi u_tdm.$$ Since $\phi$ is arbitrary, the second inequality means that
$\frac 4{NK}\partial_t u_t<\left(1+\frac {\pi^2}{K^2t^2}\right)u_t$  a.e. on $M$. Note that $\partial_tu_t$ and $u_t$ are continuous by Proposition \ref{prop21*}. We obtain (\ref{u-t-der}) on $M$. Since the LHS of (\ref{log-sob-ineq-1}) is positive, its RHS must be nonnegative, i.e., (\ref{imp-Li-Yau-ineq}) holds.

When $K=0$, it follows from Proposition \ref{prop52} that
$1+(t-s)\zeta\int_M\phi u_t\Delta(\log u_t)dm\geq 0.$ From this and Lemma \ref{lem41}, one obtains
\beqn \int_M\phi \left(\partial_t u_t-u_tF^2(\nabla \log u_t)\right)dm\geq -\frac {N}{2(t-s)}\int_M\phi u_t dm\eeqn for any nonnegative $\phi\in C^\infty(M)$. This implies that
$F^2(\nabla\log u_t)-\partial_t(\log u_t)\leq \frac {N}{2(t-s)} $ a.e. on $M$ for any $0\leq s<t<\infty$. Since $F(\nabla \log u_t)$ and $\partial_t(\log u_t)$ are continuous, we have $F^2(\nabla\log u_t)-\partial_t(\log u_t)\leq \frac {N}{2(t-s)}$ on $M$, which gives (\ref{LY-0}) by letting $s=0$. \end{proof}

 Similarly, under the same assumptions as in Theorem \ref{thm41},  we still have the conclusions of Propositions \ref{prop51}-\ref{prop52} and Theorem \ref{thm51}. Their proofs are similar provided that we use Lemma \ref{lem43} instead of Lemma \ref{lem42*} (cf. the arguments in Section 4). We omit the statements and proofs of the corresponding results.

 \bigskip

{\it Proof of Theorem \ref{thm12}}. We denote $u(t, \cdot)$ by $u_t$ as before. Note that $\Psi_t(x)$ is smooth and strictly concave on $\left(-\infty, 1+\pi^2/(K^2t^2)\right)$ (cf. Lemma 2.8, \cite{BBG}). Moreover, $$\lim\limits_{x\rightarrow -\infty}\Psi_t(x)=-\lim\limits_{x\rightarrow 1+\frac{\pi^2}{K^2t^2}}\Psi_t(x)=+\infty, \ \ \ \ \ \Psi_t(1)=\frac 1t+\frac {K}2>0, $$  which means that
  $\Psi_t$ admits exactly one root $\chi_0\in (1, 1+\frac{\pi^2}{K^2t^2})$ for any $t>0$.  Note that we should use $-K(K>0)$ instead of the lower bound $K$ of Ric$_N$ in Theorem \ref{thm51} from the assumption. Since the LHS of (\ref{imp-Li-Yau-ineq}) is nonnegative and $\int_M\phi u_tdm>0$, we have $\Psi_t(\chi)\geq 0$, which implies that $\chi\leq \chi_0$, i.e.,
   \beqn -\frac 4{NK}\int_M\phi\partial_t u_tdm\leq \chi_0\int_M\phi u_tdm. \eeqn Since $\phi$ is arbitrary, the above inequality means that $-\frac 4{NK}\partial_tu_t\leq \chi_0 u_t$ a.e. on $M$. By the continuity of $u_t$ and $\partial_tu_t$, we have $\partial_t(\log u_t)\geq -\frac{NK}4\chi_0$ on $M$.
  Further, following the elementary arguments as in \S 5.2 of \cite{BBG}, we have
\beq \Psi_t(x)\leq -\frac {\alpha K}2 x+\frac{\alpha^2}t+\frac{K\alpha^2}{2(\alpha-1)} \label{psi-K}\eeq for any constant $\alpha>1, t>0$ and $x<1+\pi^2/(t^2K^2)$.  From (\ref{imp-Li-Yau-ineq}) and (\ref{psi-K}), one obtains that
\beqn \int_M\phi u_tF^2(\nabla \log u_t)dm&\leq &\frac N 2\left\{-\frac {\alpha K}2\chi+\frac{\alpha^2}t+\frac{K\alpha^2}{2(\alpha-1)}\right\}\int_M\phi u_tdm\nonumber \\
 &=& \alpha\int_M\phi u_t\left(\partial_t(\log u_t)+\frac{N\alpha}{2t}+\frac{NK\alpha}{4(\alpha-1)}\right)dm, \eeqn
where we used $-K$ instead of $K$ in the definition of $\chi(t)$.  Since $\phi$ is arbitrary and $u_t>0$, the above inequality means that
\beqn  F^2(\nabla \log u_t) \leq \alpha\partial_t(\log u_t)+\frac{N\alpha^2}{2t}+\frac{NK\alpha^2}{4(\alpha-1)} \eeqn a.e. on $M$.  By continuity, the above inequality holds on $M$.\hfill$\Box$

\bigskip
In the case of positive curvature, i.e., Ric$_N\geq K>0$, observe that
\beqn \lim\limits_{x\rightarrow -\infty}\Psi_t(x)=\lim\limits_{x\rightarrow 1+\frac{\pi^2}{K^2t^2}}\Psi_t(x)=-\infty, \ \ \ \ \Psi_t(0)>0, \ \ \ \ \Psi_t(1)=\frac 1t-\frac K2<0\eeqn for any $t\geq 2/K$, which means that
  $\Psi_t$ admits exactly two roots $-\infty<\chi_1<0<\chi_2<1$ by the concavity of $\Psi_t$. Since $\Psi_t(\chi)\geq 0$ by (\ref{imp-Li-Yau-ineq}), we have
\beqn \chi_1\leq \chi\leq \chi_2<1\eeqn for any $t\geq 2/K$. There are quantitative estimates for $\chi_1$ and $\chi_2$ by Lemma 4.2 in \cite{BBG}. From this and the proof of Corollary 4.4 in \cite{BBG}, we have
$$\Psi_t(\chi)\leq 3K{\rm e}^{2-2Kt}$$ if $t\geq \frac 6K$. Thus, by (\ref{imp-Li-Yau-ineq}), one obtains
\beqn \int_M\phi u_tF^2(\nabla \log u_t)dm\leq \frac 32 NK {\rm e}^{2-2Kt}\int_M\phi u_tdm.\eeqn Thus $u_tF^2(\nabla \log u_t)\leq\frac 32 NK {\rm e}^{2-2Kt}u_t$ a.e. on $M$. By continuity and $u_t>0$, one obtains the following result.
\begin{cor}  \label{cor51} Let $(M, F, m)$ and $u(t, \cdot)$ be the same as in Theorem \ref{thm11} or Theorem \ref{thm41}. Assume that Ric$_N\geq K>0$.  Then there are real numbers $\chi_1<0<\chi_2<1$ such that
\beqn  \frac 14 {NK}\chi_1\leq\partial_t(\log u)\leq \frac 14 {NK}\chi_2 \eeqn for any $t\geq 2/K$, and
\beq F^2(\nabla \log u)\leq \frac 32NK{\rm e}^{2-2Kt} \label{DE**}\eeq for any $t\geq \frac 6K$. In particular, $F^2(\nabla \log u)$ is bounded when $t\geq \frac 6K$. \end{cor}

\section{Equivalent characterizations of Ric$_\infty\geq K$}

In this section, we give several equivalent characterizations of Ric$_\infty\geq K$ by means of the linearized heat semigroup. As an application, we obtain an (exponential) decay or growth of the Lipschitz constant along the heat flow depending on the sign of $K$.
Let us first recall Ohta's characterizations of Ric$_\infty\geq K$ in \cite{Oh1} or \cite{Oh2}.
\begin{prop} [\cite{Oh2}] \label{prop61} Let $(M, F, m)$ be an $n$-dimensional Finsler measure space.
Then, for each $K\in \mathbb R$, the following statements are equivalent.

(1) Ric$_\infty\geq K$.

(2) The improved Bochner inequality
$$\Delta^{\nabla h}\left(\frac{F^2(\nabla h)}2\right)-d(\Delta h)(\nabla h)\geq KF^2(\nabla h)+d(F(\nabla h))(\nabla^{\nabla h}(F(\nabla h)))$$ holds on $M_h$ for all $h\in C^\infty(M)$.

(3) The Bochner inequality

\beqn \Delta^{\nabla h}\left(\frac{F^2(\nabla h)}2\right)-d(\Delta h)(\nabla h)\geq KF^2(\nabla h) \label{BF-inf}\eeqn  holds on $M_h$ for all $h\in C^\infty(M)$.
\end{prop}

In fact, (1)$\Rightarrow$ (2) follows from Proposition 3.5 in \cite{Oh2}. It is obvious that (2)$\Rightarrow$ (3). (3)$\Rightarrow $ (1) follows from the same proof of (II)$\Rightarrow $(I) of Theorem 3.9 in \cite{Oh2}. Next we give further characterizations of Ric$_\infty\geq K$ from the viewpoint of functional inequalities. We first consider the compact case as before.

\begin{thm} \label{thm61} Let $(M, F, m)$ be an $n$-dimensional compact Finsler measure space and $u_t$ be a positive global solution to $\partial_t u=\Delta u$ (in the weak sense) on $[0, \infty)\times M$ with $u_0: =f\in C^2(M)$. Then, for any $0\leq s<t<\infty$ and $K\in \mathbb R$, the following statements are equivalent.

(1) Ric$_\infty\geq K$.

(2) It holds on $M$ that
\beq \frac {e^{-2K(t-s)}-1}{2K}P_{s, t}^{\nabla u}\left(u_sF^2\left(\nabla \log u_s\right)\right)&\leq & u_t\log u_t-P_{s, t}^{\nabla u}\left(u_s\log u_s\right)\nonumber \\
&\leq & \frac {1-e^{2K(t-s)}}{2K}u_tF^2\left(\nabla \log u_t\right), \label{log-inf}\eeq where the LHS of (\ref{log-inf}) should be read as its limit $-(t-s)P_{s, t}^{\nabla u}\left(u_sF^2\left(\nabla \log u_s\right)\right)$ and the RHS of (\ref{log-inf}) should be read as its limit $-(t-s)u_tF^2\left(\nabla \log u_t\right)$ when $K=0$.

(3) It holds the gradient estimate for $\log u_t$ on $M$
\beq u_tF^2(\nabla \log u_t)\leq e^{-2K(t-s)}P_{s,t}^{\nabla u}\left(u_sF^2(\nabla \log u_s)\right). \label{inf-ineq1} \eeq  In particular, $u_tF^2(\nabla \log u_t)\leq e^{-2Kt}P_{0,t}^{\nabla u}\left(fF^2(\nabla \log f)\right)$.

(4) It holds the gradient estimate for $u_t$ on $M$
\beq F^2(\nabla u_t)\leq e^{-2K(t-s)}P_{s,t}^{\nabla u}\left(F^2(\nabla u_s)\right).\label{inf-ineq2} \eeq  In particular, $F^2(\nabla u_t)\leq e^{-2Kt}P_{0, t}^{\nabla u}\left(F^2(\nabla f)\right)$.
\end{thm}
\begin{proof}  First of all,  the assumption that $0<u_0=f\in C^2(M)$ ensures $F^2(\nabla f)\in C^1(M)\subset H^1(M)$, which means that $fF^2(\nabla \log f)\in C^1(M)\subset H^1(M)$. Since $u_t=P_{0, t}^{\nabla u}(f)$ is a positive global solution to $\partial_t u_t=\Delta u_t$, we have $u_t\in H^2(M)\cap C^{1, \beta}(M)$ and $F(\nabla u_t)\in H^1(M)\cap C^{\beta}(M)$ by Proposition \ref{prop21*}. Thus $u_t$ and $F(\nabla u_t)$ are bounded on $M$.  Consequently, $u_t\log u_t\in H^1(M)\cap C^{1, \beta}(M)$ and  $u_tF^2(\nabla\log u_t)\in H^1(M)\cap C^\beta(M)$. By Proposition \ref{prop22}, $P_{s, t}^{\nabla u}\left(F^2(\nabla u_s)\right)$, $P_{s,t}^{\nabla u}\left(u_s\log u_s\right)$ and $P_{s,t}^{\nabla u}\left(u_sF^2(\nabla \log u_s)\right)$ are in $H^1(M)\cap C^\beta(M)$ and hence they are H\"older continuous.

 Now we prove that (1) $\Rightarrow$ (2). Assume that Ric$_\infty\geq K$. For any $\phi\in C^\infty(M)$, let $H(\sigma)$ be defined by (\ref{H-sigma}) for $\sigma\in [0, t]$. By Lemma \ref{lem42*},  we have
  $H''(\sigma)\geq - 2KH'(\sigma)$, which means that $H'(\sigma)+2KH(\sigma)\leq H'(t)+2KH(t)$ for any $\sigma \leq t$.
Thus
\beqn \frac d{d\sigma} \left(e^{2K\sigma}H(\sigma)\right)=e^{2K\sigma}\left(H'(\sigma)+2KH(\sigma)\right)\leq e^{2K\sigma}\left(H'(t)+2KH(t)\right).\label{eq0} \eeqn

{\it Case 1.} $K\neq 0$. For any $0\leq s<\sigma \leq t$,  integrating the above inequality in $\sigma$ from $s$ to $t$ yields
\beq H(t)\leq \frac 1{2K}\left(e^{2K(t-s)}-1\right)H'(t)+H(s). \label{H-eHH} \eeq
By Lemma \ref{lem42*}, we have $H(s)=\int_M \phi P_{s, t}^{\nabla u}(u_s\log u_s)dm$ and \beqn H(t)=\int_M \phi u_t\log u_tdm, \ \ \ H'(t)= - \int_M\phi u_t F^2(\nabla \log u_t)dm.\eeqn
 Plugging these into (\ref{H-eHH}) and using arbitrariness of $\phi$ yield
\beq  u_t\log u_t-P_{s, t}^{\nabla u}(u_s\log u_s)\leq \frac 1{2K}\left(1-e^{2K(t-s)}\right)u_tF^2(\nabla \log u_t) \label{log*}\eeq  a.e. on $M$.
Since each term in (\ref{log*}) is continuous, (\ref{log*}) holds on $M$.

{\it Case 2.} $K=0$. In this case, we have $H''(\sigma)\geq 0$. Then $H'(\sigma)\leq H'(t)$ for any $\sigma\leq t$. Integrating this in $\sigma$ from $s$ to $t$ gives $H(t)\leq H(s)+H'(t)(t-s)$. Similar to Case 1,  we have
\beq u_t\log u_t-P_{s, t}^{\nabla u}(u_s\log u_s)\leq -(t-s) u_tF^2(\nabla \log u_t) \label{log-inf-0}\eeq on $M$.
Obviously, the RHS of (\ref{log-inf-0}) is the limit of that of (\ref{log*}) as $K\rightarrow 0$. Thus the second inequality in (\ref{log-inf}) holds.

Similarly,  we have $H'(\sigma)+2KH(\sigma)\geq H'(s)+2KH(s)$ for any $0\leq s<\sigma\leq t$ by the monotonicity. Hence
\beq \frac d{d\sigma} \left(e^{2K\sigma}H(\sigma)\right)=e^{2K\sigma}\left(H'(\sigma)+2KH(\sigma)\right)\geq e^{2K\sigma}\left(H'(s)+2KH(s)\right).\label{eq3}\eeq
Integrating this in $\sigma$ from $s$ to $t$ yields
\beqn H(t)\geq \frac 1{2K}\left(1-e^{-2K(t-s)}\right)H'(s)+H(s) \label{ineq4} \eeqn when $K\neq 0$.  From this and Lemma \ref{lem42*}, one obtains
\beq  u_t\log u_t-P_{s, t}^{\nabla u}(u_s\log u_s)\geq \frac 1{2K}\left(e^{-2K(t-s)}-1\right)P_{s, t}^{\nabla u}\left(u_sF^2(\nabla \log u_s)\right) \label{rev-log*}\eeq  a.e. on $M$.
Since each term in (\ref{rev-log*}) is continuous, (\ref{rev-log*}) holds on $M$.

 For the case when $K=0$, since $H''(\sigma)\geq 0$, we have $H'(\sigma)\geq H'(s)$ for $\sigma>s$. Integrating this in $\sigma$ from $s$ to $t$ yields $H(t)\geq H(s)+H'(s)(t-s)$, which implies that
 \beq u_t\log u_t-P_{s, t}^{\nabla u}(u_s\log u_s)\geq -(t-s) P_{s, t}^{\nabla u}\left(u_sF^2(\nabla \log u_s)\right)  \label{rev-log-inf-0}\eeq on $M$.
Obviously, the RHS of (\ref{rev-log-inf-0}) is the limit of that of (\ref{rev-log*}) as $K\rightarrow 0$. Thus the first inequality in (\ref{log-inf}) holds.

\bigskip

(2) $\Rightarrow$ (3). Assume that (\ref{log-inf}) holds. If $K=0$, the claim directly follows from (\ref{log-inf}).  When $K\neq 0$, it follows from (\ref{log-inf}) that
\beqn \frac{e^{-2K(t-s)}-1}{2K}P_{s, t}^{\nabla u}\left(u_sF^2\left(\nabla \log u_s\right)\right)\leq \frac{1-e^{2K(t-s})}{2K}u_tF^2\left(\nabla \log u_t\right).\eeqn  Note that $\frac{1-e^{2K(t-s)}}{2K}<0$. The above inequality is equivalent to
\beqn u_tF^2\left(\nabla \log u_t\right)&\leq &\frac{1-e^{-2K(t-s)}}{e^{2K(t-s)}-1}P_{s, t}^{\nabla u}\left(u_sF^2\left(\nabla \log u_s\right)\right)\nonumber \\
&=&e^{-2K(t-s)}P_{s, t}^{\nabla u}\left(u_sF^2\left(\nabla \log u_s\right)\right). \eeqn Thus (\ref{inf-ineq1}) follows.

\bigskip

 (3) $\Rightarrow$ (4). Assume that  (\ref{inf-ineq1}) holds. Let $u_t: =P_{0, t}^{\nabla u}(f)$ and $\tilde u_t:=P_{0, t}^{\nabla u}(1+\varepsilon f)$ for some small $\varepsilon>0$.   Since $P_{0, t}^{\nabla u}(1)=1$,  $\tilde u_t=P_{0, t}^{\nabla u}(1+\varepsilon f)=1+\varepsilon u_t>0$. In local coordinates, $g^{*ij}(x, d\tilde u_t)=g^{*ij}(x, du_t)=g^{ij}(x, \nabla u_t)$ on $M_{\tilde u_t}=M_{u_t}$. Hence,
 \beq \nabla \tilde u_t=g^{ij}(\nabla \tilde u_t)\frac{\partial \tilde u_t}{\partial x^j}\frac{\partial}{\partial x^i}=g^{*ij}(d\tilde u_t)\frac{\partial \tilde u_t}{\partial x^j}\frac{\partial}{\partial x^i}=\varepsilon g^{*ij}(du_t)\frac{\partial u_t}{\partial x^j}\frac{\partial}{\partial x^i}=\varepsilon \nabla u_t \label{uu}\eeq on $M_{\tilde u_t}=M_{u_t}$. Obviously,  $\nabla \tilde u_t=\nabla u_t=0$ on $M\setminus M_{u_t}$. Thus $\nabla \tilde u_t=\varepsilon \nabla u_t$ on $M$. From this, we have for any $\phi\in C^\infty(M)$,
 $$\int_M \phi \partial_t\tilde u_tdm=\varepsilon\int_M \phi \partial_tu_t dm=-\varepsilon \int_M d\phi(\nabla u_t)dm=-\int_Md\phi(\nabla \tilde u_t)dm, $$ which means that $\tilde u_t$ is a global solution to the heat equation with $0<\tilde u_0=1+\varepsilon f\in C^\infty(M)$. By Taylor's expansions, we have
 $\frac 1{\tilde u_t}=\frac 1{1+\varepsilon u_t}=1-\varepsilon u_t+\varepsilon^2 u_t^2+o(\varepsilon^2)$ as $\varepsilon\rightarrow 0$. Note that $M_{\log \tilde u_t}=M_{\tilde u_t}=M_{u_t}$. Consequently,
    \beq \tilde u_t F^2(\nabla \log \tilde u_t)=\tilde u_t F^{*2}(d(\log\tilde u_t))=\tilde u_t^{-1}F^{*2}(d\tilde u_t)=\varepsilon^2F^{2}(\nabla u_t)+o(\varepsilon^2) \label{u-logu-2}\eeq as $\varepsilon\rightarrow 0$.
 From this, applying (\ref{inf-ineq1}) to $\tilde u_t=1+\varepsilon u_t$ and then letting $\varepsilon\rightarrow 0$ yield (\ref{inf-ineq2}).

\bigskip

 (4) $\Rightarrow$ (1). Assume that (4) holds. Note that (\ref{inf-ineq2}) is an equality at $t=s$. By Lemma \ref{lem41},  we have
\beqn 0&\leq &\lim\limits_{t\rightarrow s}\frac{e^{-2K(t-s)}P_{s, t}^{\nabla u}(F^2(\nabla u_s))-F^2(\nabla u_t)}{2(t-s)}\\
&=&\frac 12\left\{\Delta^{\nabla u_s}(F^2(\nabla u_s))-2KF^2(\nabla u_s)-2d(\Delta u_s)(\nabla u_s)\right\}\eeqn  on $M_{u_s}$. Taking $s=0$ yields
$$\Delta^{\nabla f}\left(\frac{F^2(\nabla f)}2\right)-d(\Delta f)(\nabla f)\geq KF^2(\nabla f)$$ on $M_f$, which implies Ric$_\infty\geq K$ by Proposition \ref{prop61}. This finishes the proof. \end{proof}

It follows from Proposition \ref{prop61} and Theorem  \ref{thm61} that
\begin{cor} Let $(M, F, m)$ be an $n$-dimensional compact Finsler measure space with Ric$_\infty\geq K (K\in\mathbb R$). Then (1)--(3) in Proposition \ref{prop61} and (1)--(4) in Theorem \ref{thm61} are equivalent to each other.\end{cor}

\begin{remark} {\rm In Theorem \ref{thm61}, the inequalities in (2)--(3) and hence (4) need the positivity of $u_t$. In fact,  (\ref{inf-ineq2}) holds for any global solution $u_t$ to $\partial_t u=\Delta u$ in the compact case. This was first proved in Theorem 4.1 of \cite{OS2} in a different way. } \end{remark}

 Recall that the Lipschitz constant of a continuous function $f$ is defined by
$${\rm{Lip}}(f)=\sup\limits_{x, z\in M}\frac{f(z)-f(x)}{d_F(x, z)}, $$ where $d_F$ is the distance function induced by $F$. Note that
$$\limsup\limits_{z\rightarrow x}\frac{f(z)-f(x)}{d_F(x, z)}=F^*(df(x))=F(\nabla f(x))$$ for any $f\in C^1(M)$ (\cite{OS2}).
By Corollary \ref{cor31}, we have $P_{0, t}^{\nabla u}(F^2(\nabla u_0))\leq \|F^2(\nabla u_0)\|_{L^\infty}\leq \|F(\nabla u_0)\|^2_{L^\infty}$. From this and (\ref{inf-ineq2}), one obtains an (exponential) decay (when $K\geq 0$) or growth (when $K\leq 0$) estimate of the Lipschitz constants along the heat flow, which was due to S.Ohta.

 \begin{cor} [\cite{Oh1}] Let $(M, F, m)$ be the same as in Theorem \ref{thm61} and Ric$_\infty\geq K (K\in \mathbb R)$. Then, for any global solution $(u_t)_{t\geq 0}$ to the heat equation with $u_0\in C^2(M)$, we have
\beq \|F(\nabla u_t)\|_{L^\infty}\leq e^{-Kt}\|F(\nabla u_0)\|_{L^\infty}. \label{bound}\eeq
and hence ${\rm{Lip}}(u_t)\leq e^{-Kt}{\rm{Lip}}(u_0) \label{Lip}$ for all $t\in [0, T]$.
\end{cor}

 For the complete and noncompact case, we have the following result.
 \begin{thm} \label{thm62} Let $(M, F, m)$ be a complete and noncompact Finsler measure space satisfying (\ref{unif cs cons}) and $u_t$ be a positive global solution to $\partial_tu=\Delta u$ (in the weak sense) with $u_0=f\in C_0^2(M)$ and $F\big(\nabla[F(\nabla u_t)]\big)\in L^2(M)$. Assume that Ric$_\infty\geq K (K\in\mathbb R$). Then (1)--(3) in Proposition \ref{prop61} and (1)--(4) in Theorem \ref{thm61} are equivalent to each other. In particular, (\ref{bound}) holds on $M$ for $t>0$. \end{thm}

\begin{proof}  We first remark that (\ref{unif cs cons})  guarantees the existence of semigroup $\{P_{s,t}^{\nabla u}\}$. In this case,  $H_0^1(M)=H^1(M)$ as in the proof of Lemma \ref{lem43}. Moreover, $f\in C_0^2(M)$ ensures that $f$ and $F(\nabla f)$ are in $H^1(M)\cap L^\infty(M)$. By Propositions \ref{prop21*}--\ref{prop22} and Theorem \ref{thm31},   $u_t=P_{0, t}^{\nabla u}(f)\in H^1(M)\cap C^{1, \beta}(M)\cap L^\infty(M)$ with $u_t\in H_{loc}^2(M)$ and $P_{0, t}^{\nabla u}(F(\nabla f))\in H^1(M)\cap C^\beta(M)\cap L^\infty(M)$. Consequently, both $u_t\log u_t$ and $P_{s, t}^{\nabla u}(u_s\log u_s)$ are in $H^1(M)\cap C^\beta(M)$.

Since $F\big(\nabla[F(\nabla u_t)]\big)\in L^2(M)$ by the assumption,  we have $F(\nabla u_t)\in H^1(M)$.
By (\ref{unif cons}), the assumption that $F\big(\nabla[F(\nabla u_t)]\big)\in L^2(M)$ is equivalent to $d[F(\nabla u_t)]\big(\nabla^{\nabla u_t}[F(\nabla u_t)]\big)\in L^1(M)$. In fact, it follows from $$\kappa^{-1}F^2(\nabla h)\leq dh(\nabla^{\nabla u}h)=g_{\nabla u}\big(\nabla^{\nabla u}h, \nabla^{\nabla u}h\big)=g^*(du)(dh, dh)\leq (\kappa^*)^{-1}F^2(\nabla h),$$ where $h:=F^2(\nabla u_t)$.
 Hence we get
\beqn F(\nabla u_t)\leq e^{-K(t-s)}P_{0, t}^{\nabla u}(F(\nabla f)) \label{Fu-Ff}\eeqn by Theorem 3.7 in \cite{Oh2}, which implies that $F(\nabla u_t)\in L^\infty(M)$. Thus  $F(\nabla u_t)\in H^1(M)\cap L^\infty(M)$ and hence $F^2(\nabla u_t)\in H^1(M)$.  Consequently, $u_tF^2(\nabla\log u_t)=u_t^{-1}F^2(\nabla u_t)\in H^1(M)$.  The rest of the proof follows from that of Theorem \ref{thm61}. This finishes the proof. \end{proof}

 \section*{ACKNOWLEDGMENT}
The author was partially supported by National Natural Science Foundation of China under Grant Nos. 12471044, 12071423 and Zhejiang Provincial Natural Science Foundation of China under Grant No. LZ26A010004. The author expresses her sincere thanks to the anonymous referees for their helpful suggestions and comments.

 \end{document}